\documentclass[12pt]{amsart}

\usepackage{cancel}
\usepackage{latexsym, amssymb, amsmath}
\usepackage{soul,esint}
\usepackage{amsfonts, graphicx}
\usepackage{graphicx,color}
\usepackage{mathtools}
\usepackage{hyperref}
\usepackage{verbatim}
\makeatletter
\renewcommand*{\eqref}[1]{%
  \hyperref[{#1}]{\textup{\tagform@{\ref*{#1}}}}%
}
\makeatother

\newcommand{\be}{\begin{equation}}
\newcommand{\ee}{\end{equation}}
\newcommand{\beq}{\begin{eqnarray}}
\newcommand{\eeq}{\end{eqnarray}}

\newtheorem{thm}{Theorem}[section]
\newtheorem{conj}{Conjecture}[section]

\newtheorem{lma}{Lemma}[section]
\newtheorem{prop}{Proposition}[section]
\newtheorem{cor}{Corollary}[section]
\newtheorem{defn}{Definition}[section]
\theoremstyle{remark}
\newtheorem{rem}{Remark}[section]
\numberwithin{equation}{section}

\def\be{\begin{equation}}
\def\ee{\end{equation}}
\def\bee{\begin{equation*}}
\def\eee{\end{equation*}}

\def\wt{\widetilde}
\def\la{\langle}
\def\ra{\rangle}
\def\p{\partial}

\def\R{\mathbb{R}}

\def\o{\overline}

\def\d{\frac {\operatorname d}{\operatorname {dt}}}

\def\vh{\vspace{.2cm}}
\def\S{\Sigma}

\def\Lp{\Delta}
\def\Na{\nabla}
\def\eps{\varepsilon}

\def\nl{\frac{\Na u}{|\Na u|}}
\def\d{\delta}

\def\O{\Omega}

\def\ti{\tilde}

\def\mL{\mathcal{L}}
\def\l{\lambda}

\def\mF{\mathcal{F}}
\def\mE{\mathcal{E}}
\def\mB{\mathcal{B}}
\def\mH{\mathcal{H}}
\def\HH{\mathbb{H}}
\begin{document}
\title[]
{Dihedral Rigidity for Cubic Initial Data Sets}

 \author{Tin-Yau Tsang}
\address [Tin-Yau Tsang] {Department of Mathematics, University of California, Irvine}
\email{tytsang@uci.edu}

%\renewcommand{\subjclassname}{
  %\textup{2010} Mathematics Subject Classification}
%\subjclass[2010]{Primary 32Q15; Secondary 53C44
%}

\date{Aug, 2021}

\begin{abstract} 
In this paper we pose and prove a spacetime version of Gromov's dihedral rigidity theorem (\cite{Gro14},\cite{Li0},\cite{Li1}) for cubes when the dimension is 3 by studying the level sets of spacetime harmonic functions (\cite{S},\cite{BS},\cite{HKK}), extending the work of \cite{CK}.  As a corollary, we also obtain an alternative proof of dihedral rigidity for prisms in hyperbolic space (\cite{Li2}). We then discuss the relation between polyhedra and the spacetime positive mass theorem. This generalises the work of \cite{MP} and \cite{Li1}. Finally, we show dihedral rigidity of charged Riemannian cubes by charged harmonic functions (\cite{BHKKZ}). 
\end{abstract}

\keywords{}

\maketitle

\markboth{TIN-YAU TSANG}{Dihedral Rigidity for Cubic Initial Data Sets}

\section{Introduction}
Gromov (\cite{Gro14} Section 2.2) proposed the following conjecture to study the geometry of scalar curvature with a lower bound and to define non-negative scalar curvature for $C^0$ metric. 
\begin{conj}[The dihedral rigidity conjecture]\label{conj:dihedral.rigidity}
Suppose $(M,g)$ is a Riemannian polyhedron with nonnegative scalar curvature and weakly mean convex faces. Suppose 
that the dihedral angles of $(M,g)$ are not larger than the (constant) dihedral angle between corresponding faces of the model Euclidean polyhedron $(M,g_{Euc})$. Then $(M,g)$ is isometric to a flat Euclidean polyhedron.
\end{conj}

\begin{comment}
Gromov considered the cube case and obtained the following result. 
\begin{thm}[\cite{Gro14}]\label{theo.cube.comparison.nonrigid}
Let $M=[0,1]^n$ be a cube, and $g$ be a Riemannian metric on $M$. Then $(M,g)$ cannot simultaneously satisfy:
\begin{enumerate}
    \item The scalar curvature of $g$ is positive;
    \item Each face of $M$ is strictly mean convex with respect to the outward normal vector field;
    \item Everywhere the dihedral angle between two faces of $M$ is acute.
\end{enumerate}
\end{thm}
\end{comment}

Chao Li has made major progress on this problem; in particular, the following results are obtained.    
\begin{thm}(\cite{Li0},\cite{Li1})\label{theo.dihedra.rigidity}
Let $2\le n\le 7$, $P^n$ be a Euclidean prism with dihedral angles at most $\pi/2$, and if $n=3$, $P^3$ can be an arbitrary simplex in $\R^3$.  Assume $M^n$ is a Riemannian polyhedron of type $P$. Then Conjecture \ref{conj:dihedral.rigidity} holds for $M$. Precisely, if $g$ is a $C^{2,\alpha}$ metric on $M$ such that
\begin{enumerate}
    \item The scalar curvature of $g$ is nonnegative;
    \item Each face of $M$ is weakly mean convex;
    \item The dihedral angles between adjacent faces of $(M,g)$ are everywhere less than or equal to the corresponding (constant) dihedral angles of $(P,g_{Euc})$.
\end{enumerate}
Then $(M,g)$ is isometric to a Euclidean polyhedron.
\end{thm}
There is also the following polyhedral comparison result for hyperbolic polyhedra.
\begin{thm}(\cite{Li2})\label{theo.parabolic.cube}
Let $2\le n\le 7$. In a coordinate system such that $g_\HH$ is expressed as $$(dx^1)^2+e^{2x^1}\sum_{j=2}^{n}(dx^j)^2,$$ assume $M^n$ is a Riemannian polyhedron modelled on $[0,1]\times P^{n-1}$, where $P^{n-1}$ is a polyhedron such that Theorem \ref{theo.dihedra.rigidity} holds.  Denote the face $\partial M\cap \{x^1=1\}$ by~$F_T$ and the face $\partial M\cap \{x^1=0\}$ by~$F_B$. Assume~$g$ is a $C^{2,\alpha}$ Riemannian metric on~$M$ such that
\begin{enumerate}\itemsep=0pt
\item$R(g)\ge -n(n-1)$ in $M$;
\item $H(g)\ge n-1$ on~$F_T$, $H(g)\ge -(n-1)$ on~$F_B$, and~$H(g)\ge 0$ on~$\partial M\setminus (F_T\cup F_B)$;
\item The dihedral angles between adjacent faces of $(M,g)$ are everywhere less than or equal to the corresponding (constant) dihedral angles of $([0,1]\times P^{n-1},g_{\HH})$.
\end{enumerate}
Then $(M,g)$ is isometric to a parabolic prism in $\HH^n$.
\end{thm}

From the perspective of initial data sets, the aforementioned results give a comparison of a given polyhedral initial data set to standard ones $(\R^n,g_{Euc},0)$ and $(\HH^n,g_{\HH},g_{\HH})$.  Meanwhile, the lower bounds on scalar curvature and mean curvature correspond to energy conditions.  A natural question arises as to whether there is a corresponding version of these dihedral rigidity results for general initial data sets.

Stern (\cite{S}) proposed a novel harmonic function approach to study scalar curvature and the topology of 3 dimensional manifolds.  Subsequently the method was extended to give new proofs of several positive mass theorems 
using ``spacetime" harmonic functions (\cite{BHKKZ}).  Recently, Chai and Kim (\cite{CK}) adopted the harmonic function approach to prove dihedral rigidity for cubes. This involved prescribing suitable boundary conditions for the harmonic functions. Following their ideas, and motivated by the Hamiltonian formulation of the Einstein equations, we have obtained results on dihedral rigidity for general initial data sets satisfying natural energy conditions.

\begin{thm}\label{maindihedral}
Let $(M^3,g,k)$ be an initial data set of cube type which simultaneously satisfies: 
\begin{enumerate}
\item the dominant energy condition, 
\item the boundary dominant energy condition, 
\item everywhere the dihedral angle between two faces of $M$ is less than or 
equal to $\pi/2$. 
\end{enumerate}
Then,  $(M,g,k)$ can be isometrically embedded into Minkowski space with boundary isometric to the boundary of a Euclidean rectangular prism.  
\end{thm}

As a corollary we obtain the following result. 
\begin{cor}\label{non-existence}(cf. \cite{Gro14})
Let $(M^3,g,k)$ be an initial data set of cube type.  Then $(M,g,k)$ cannot simultaneously satisfy: 
\begin{enumerate}
\item the dominant energy condition, 
\item the boundary dominant energy condition, 
\item all dihedral angles of $M$ are acute. 
\end{enumerate}
\end{cor}

It is shown in \cite{Gro14} %P.1144 
Section 4.9 that there exists a mean convex cubical domain with negative scalar curvature and strictly acute dihedral angles.  Hence,  Corollary \ref{non-existence} can be seen as a precise local characterization of the dominant energy condition.  

We now outline the proof.  First, we consider a solution $u$ to the following mixed boundary value problem.  
\begin{lma}\label{PDEcube2}
Given an initial data set $(M^3,g,k)$ of type $P$, where all dihedral angles are everywhere smaller than $\pi$,  then there exists a non-negative spacetime harmonic function $u\in C^{0,\alpha}(M) \cap C^{1,\alpha}_{loc}(M\setminus (\bar T \cup \bar B)) \cap C^{2,\alpha}_{loc}(M\setminus \bar \mE)\cap W^{3,p}_{loc}(\mathring M)$ such that 
\begin{enumerate}
\item $G_0(u):=\Lp u + K|\Na u|=0$ in $\mathring M$,
\item $u=0$ on $B$ and $u=1$ on $T$,
\item $\p_{\nu} u =0$ on $F$,
\end{enumerate}
where $K:=tr_gk$,  $\nu$ denotes the outward unit normal of $\p M$; $T, B, F$ and $\mE$ denote the top, the bottom, the side faces and the edges of $M$ respectively. 
\end{lma} 

%\tb{Actually, we have proved "slightly" more than that, $K$ can be any smooth enough functions actually.}

The proof of Lemma \ref{PDEcube2} is based on a regularization and application of the implicit function theorem.  Then, under the assumptions of Theorem \ref{maindihedral}, we can see that $M$ is smoothly foliated by level sets of $u$.  In particular,  we show $M$ is foliated by stable free boundary MOTS.  We then apply the results of \cite{ALY2} Section 5 to study each level set using certain integral formulae for spacetime harmonic functions (\cite{HKK},\cite{HMT},\cite{CK},\cite{T}). Then, the flow generated by $\frac{\Na u}{|\Na u|^2}$ on $M$ is studied.  Finally, we can conclude the proof using the geometric
assumptions on $M$.  

As a corollary, we obtain an alternative proof of dihedral rigidity of parabolic prisms in hyperbolic space in the $3$ dimensional case. We note that the proof of Theorem \ref{theo.parabolic.cube} (Theorem 2.4 in \cite{Li2}) (which allows more general topology and works in dimension up to $7$) shows that $M$ is densely foliated by free boundary horospheres which minimize a certain functional and deduces the isometry by studying the Riemann curvature tensor.  With our trivial topological assumption on $M$, we can show that $M$ is smoothly foliated by free boundary horospheres and we can conclude the isometry by writing the metric in a split form and studying its properties.  

\

In terms of relating energy and geometry,  Theorem \ref{maindihedral} is for compact polyhedra while for asymptotically flat initial data sets, we have the spacetime positive mass theorem.  By recent works of Miao (\cite{M2}) and Miao and Piubello (\cite{MP}), which link dihedral angles to ADM energy, we have the following result which connects the local and the global pictures.  

\begin{thm}
Let $(M^n, g, k)$ be an asymptotically flat initial data set satisfying the dominant energy condition.  Let $\{P_k\}$ denote a sequence of Euclidean polyhedra satisfying
conditions in \cite{MP} Theorem 1.1.  Let $\vec{a}=a^i {\p_i}$, where $\sum_{i=1}^n (a^i)^2=1$, then %Let $X$ be a sequence of vector fields on $M$ with $\lim X_k|_{\mF(\p P_k)}=\p_1$.  
\be
\begin{split}
\lim_{k\to\infty}\left( -\ \int_{\mF(\p P_k)} H  \,d \sigma+\ \int_{\mF(\p P_k)}\pi(\vec{a},\nu) \,d \sigma+\ \int_{\mE(\p P_k)} ( \alpha - \bar \alpha) \, d \mu \right) \ge 0, 
\end{split} 
\ee 
where $\mF(\p P_k)$ and $\mE(\p P_k)$ denote the faces and edges of $\p P_k$ respectively, $\nu$ is the outward unit normal with respect to $g$ on corresponding faces, $\alpha$ and $ \bar \alpha$ respectively denote the dihedral angles with respect to $g$ and $g_{Euc}$. 
\end{thm}

Moreover, similar to an observation in \cite{Li1}, with Lohkamp's construction of $(\mu-|J|_g)>0$-islands (\cite{Lohkamp}), we can see dihedral rigidity implies the spacetime positive mass theorem. 

\

Finally, we show the dihedral rigidity of Riemannian cubes with charge by considering charged harmonic functions (\cite{BHKKZ}) and the properties of their electric fields. 
\begin{thm}\label{chargedEuclideanprism}
Let $(M^3,g,\mE)$ be a charged initial data set of type $P$, where $P_0$ is a rectangle, which simultaneously satisfies: 
\begin{enumerate}
\item the charged dominant energy condition, 
\item the charged boundary dominant energy condition,  
\item everywhere the dihedral angles between two faces of $M$ are less than or equal to $\pi/2$, 
\end{enumerate}
where $H$ is computed with respect to $\nu$, the unit outward normal of $M$.  Then,  $(M,g)$ is conformally equivalent to a Euclidean rectangular prism. Furthermore, $(M,g,\mE)$ can be isometrically embedded into the time slice of a Majumdar-Papapetrou spacetime. 
\end{thm}

\

This text is organized as follows. In Section \ref{Preliminaries}, some definitions about initial data sets and Hamiltonian formulation are reviewed.  Integral formulae of spacetime harmonic functions on prisms are proved in Section \ref{Integral Formulae}. Dihedral rigidity for cubes in general initial data sets and prisms in hyperbolic spaces are proved in Section \ref{Dihedral Rigidity}.  The relation among polyhedra, dihedral rigidity and the spacetime positive mass theorem is discussed in Section \ref{localized polyhedra SPMT}. In Section \ref{Charged Riemannian cubes}, dihedral rigidity for charged Riemannian cubes is discussed. The proof of Lemma \ref{PDEcube2} will be given in Appendix \ref{Existence of spacetime harmonic functions 2} after a special case is discussed in Appendix \ref{Existence of spacetime harmonic functions}.  

\

\textbf{Acknowledgements.} 
The author would like to thank Xiaoxiang Chai for explaining \cite{CK} in detail. The author is grateful to Chao Li for patiently answering his questions. The author would also like to thank Demetre Kazaras for explaining \cite{HKK} clearly and Sven Hirsch for suggesting the problem of charged polyhedra and helpful comments on an earlier version of this manuscript. 

The author thanks Pak-Yeung Chan, Man-Chun Lee and Long-Sin Li for very helpful discussions.  The author would like to thank Prof. Pengzi Miao for kindly answering his questions and his expertise in hyperbolic spaces. Furthermore, the author wants to express gratitude to Prof. Martin Man-Chun Li, Prof. Connor Mooney, Prof. Richard Schoen and Prof. Luen-Fai Tam for their insight into the problem.  

\section{Preliminaries}\label{Preliminaries}
\begin{defn}(cf. \cite{Li0} Definition 1.1, \cite{Li1} Definition 1.4, 1.5 and \cite{Li2} Definition 2.1, 2.2) 
Let $P_0\subset \R^2$ be a convex Euclidean polygon and $P:=[0,1] \times P_0$.  An initial data set $(M^3,g,k)$ with non-empty boundary, where $g$ is a $C^{2,\alpha}(M)$ metric and $k$ is a $C^{1,\alpha}(M)$ symmetric (0,2)-tensor, is said to be of type $P$ if $M$ admits a Lipschitz diffeomorphism $\Psi: M\to P$ such that $\Psi^{-1}$ is smooth when restricted to the interior, the faces and the edges of $P$. Furthermore, 

\begin{enumerate}
\item $(P, g_{Euc})$  is called a Euclidean prism.  
\item Under the coordinate for $(\HH^3,g_{\HH})$ in Theorem \ref{theo.parabolic.cube}, $(P, g_{\HH})$ is called a parabolic prism.  
\end{enumerate}
\end{defn}

\begin{defn}\label{DEC}
Under constraint equations, we can define the mass density $\mu$ and the current density $J$ by
$$\mu=\frac{1}{2}(R_g+(tr_g k)^2-|k|_g^2)$$
and $$J=div_g \pi,$$ 
where $\pi=k-(tr_g k)g$ is the conjugate momentum tensor.  
An initial data set $(M,g,k)$ is said to satisfy the dominant energy condition if 
$$\mu\geq|J|_g.$$ 
\end{defn}

\begin{defn}\label{BDEC}(cf. \cite {ADLM} Definition 2.3)
An initial data set $(M,g,k)$ is said to satisfy the boundary dominant energy condition if 
$$H\geq|\pi(\cdot,\nu)|$$ 
on $\p M$, where the mean curvature $H$ is computed with respect to the unit outward normal $\nu$. 
\end{defn} 

\begin{defn}
A hypersurface $S\subset M$ is called a marginally outer trapped surface ($MOTS$) if on $S$, the outer null expansion 
$$\theta_+:=H+tr_{S}k=0;$$ 
a marginally inner trapped surface ($MITS$) if on $S$, the inner null expansion 
$$\theta_-:=H-tr_{S}k=0.$$ 
\end{defn}

\begin{defn} Let $n\geq 3$, an initial data set $(M^n, g, k)$ is called asymptotically flat if there exists a compact set $\mathcal{C}\subset M$ such that  $M\setminus \mathcal{C}=\coprod_{i=1}^k N_i$, where each end $N_i=\R^n\setminus B_{r_i}$ by a coordinate diffeomorphism under which
$$
g_{ij} = \delta_{ij} + O^2(|x|^{-q}),
$$
and 
$$
k_{ij}=O^1(|x|^{-q-1}),
$$
where $q>\frac{n-2}{2}$, $ \mu, J \in L^1(M)$ and for a function $f$ on $M$, $f = O^{m}(|x|^{-p})$ means $\sum_{|l|=0}^m||x|^{p+|l|}\partial^l f|$ is bounded near the infinity.

\

For each end, the ADM energy-momentum vector $(E,P)$ and the ADM mass $\mathfrak{m}$ \cite{ADM} are given by 
$$
E = \frac{1}{2c(n)} \lim_{r \to \infty} \int_{|x| = r}  ( g_{ij,i} - g _{ii, j} ) \nu^j,
$$ 
$$P_i:=\frac{1}{c(n)} \lim_{r \to \infty} \int_{|x| = r}  \pi_{ij} \nu^j, \,\ \,\ \,\ i=1,...,n,$$ 
and 
$$\mathfrak{m}=\sqrt{E^2-|P|^2},$$
where the outward unit normal $\nu$ and surface integral are with respect to the Euclidean metric and $c(n)=(n-1)|\mathbb{S}^{n-1}|$. 
\end{defn}

\begin{thm}\label{SPMT}(Spacetime positive mass theorem) 
Let $n\geq 3$ and let $(M^n,g,k)$ be an asymptotically flat initial data set that satisfies the dominant energy condition. Then 
$$E\geq|P|.$$
\end{thm} We refer readers to \cite{EHLS} and \cite{W} (\cite{PT}) for its proof. 

\begin{defn}(\cite{HKK} Section 3)
The spacetime Hessian tensor is defined by 
\be 
\o \Na\o \Na u=\Na \Na u +|\Na u|k. 
\ee 

A function $u$ on $M$ is called spacetime harmonic if 
\be 
\o \Lp u:=tr_g \o\Na \o\Na u=\Lp u+(tr_g k)|\Na u|=0.
\ee
\end{defn}

\subsection{Hamiltonian formulation (Hamilton-Jacobi analysis)}
Let $(\Omega^n,g,k)$ be a compact initial data set with boundary $\S$.  A spacetime $(N^{n+1},\bar g)$ with boundary $\bar \S$ can be constructed by infinitesimally deforming the initial data set $(\Omega,g,k,\S)$ in a transversal, timelike direction $\p_t= V \vec{n} + W^i\p_i$ which satisfies $\bar\Na_{\p_t}\p_t=1$, where $V$ is the lapse function, $\vec{n}$ is the timelike unit normal of $\Omega$ in $N$ and $W$
is the shift vector. Further assume that $\Omega$ meets
$\bar \S$ orthogonally. The purely gravitational contribution $\mathcal{H}_{grav}$ to the total
Hamiltonian at the slice $\Omega$ is given by (\cite{ADM},\cite{RT},\cite{BY},\cite{HH})
\be \label{Ham}
c(n) \mathcal{H}_{grav}(V,W) = \int_{\Omega} (\mu V  + \la J,W \ra) - \int_{\S}(HV-\pi(\nu,W)),
\ee
where $H$ is the mean curvature of $\S$ with respect to the outward normal of $\Omega$ and $\pi$ is the conjugate momentum tensor.  From this, we can expect that the contribution to the boundary geometry is from the mean curvature $H$ and the 1-form $\pi(\nu,\cdot)$ An interesting difference between \eqref{Ham} and \eqref{CUBE} is that they are related to a timelike vector field and a null vector field respectively.  

\section{Integral Formulae}\label{Integral Formulae}
The following integral formula links the interior energy condition and the boundary behaviour of an initial data set. 
\begin{lma}\label{IFC}(cf. \cite{HKK} Proposition 3.2) 
Let $(M^3,g,k)$ be an initial data set of type $P$, where all dihedral angles are everywhere smaller than $\pi$. Further assume that the dihedral angles between $T$ and $F$ and those of $B$ and $F$ are everywhere less than or equal to $\pi/2$.  Then,  for a spacetime harmonic function $u$ in Lemma \ref{PDEcube2}, 
\be\label{base}
\begin{split}
&\ \int_{M} \frac12  \frac{|\o \Na \o \Na u|^2}{| \nabla u |} + \mu|\Na u| + \la J, \Na u\ra  \, d V\\
\le & \ \int_{\p_{\neq0} M} \p_\nu | \nabla u | \, d \sigma +\int_{\p M} k(\Na u, \nu) d \sigma+ \frac{1}{2}\int_{0}^{1} \int_{\S_t} R_{\S_t} dA  d t, 
\end{split} 
\ee
where $\p_{\neq0}M=\{x\in \p M\,\ |\,\ |\Na u| \neq 0\}$, $\S_t=\{u=t\}$ and $\nu$ is the outward unit normal on $\p M$.  
\end{lma}
\begin{proof}
We here assume that $|\Na u|\neq 0$ for the simplicity of presentation.  For the full generality, one should first consider $\sqrt{|\Na u|^2+\delta^2}$ for $\delta>0$ and then take limit as $\d\to 0$ (see \cite{S},\cite{BS},\cite{HKK},\cite{HMT} $Remark$ 3.3).  

\

It suffices to verify the divergence theorem such that the following holds since the remaining would be the same as in the references aforementioned (application of Bochner formula, Gauss equation and coarea formula). 
\be\label{divergence}
\begin{split}
\int_{\p M} \p_{\nu} |\Na u| = \int_{M} \Lp |\Na u|. 
\end{split}
\ee

\

Let $\{M_r\}_{r>0}$ be an exhaustion of $M$ with vertices and edges of $M$ being smoothed out, where $r$ is the parameter of radius of spherical cap around the vertices and rounded-off cylinders along the edges. The functions are regular enough on $M_r$ so that the divergence theorem can be applied. 
\be\label{exhaustion}
\begin{split}
\int_{\p M_r} \p_{\nu_r} |\Na u| = \int_{M_r} \Lp |\Na u|. 
\end{split}
\ee

{\bf 1.} \underline{L.H.S. of \eqref{divergence}.} To show that $\int_{\p M} \p_{\nu} |\Na u|$ is well-defined, we consider the following. 
\begin{prop}\label{BoundaryTerms}(cf.  \cite{HMT} Proposition 2.2,  \cite{CK},\cite{T})
Let $\S$ be a face of a type $P$ initial data set $(M^3,g,k)$,  for $|\Na u|>0$,
\begin{enumerate}
\item
if $u=constant$ on $\Sigma$, 
\be
\begin{split}
 \p_\nu | \nabla u  |= -K\nu(u)- H |\Na u|; 
\end{split}
\ee
\item
if $\p_{\nu}u=0$ on $\S$, 
\be
\begin{split}
 \p_\nu | \nabla u |
 =& -|\Na u| \Pi (\frac{\Na u}{|\Na u|}, \frac{\Na u}{|\Na u|}),
 \end{split}
\ee
where $\Pi$ is the second fundamental form with respect to the outward normal $\nu$.  
\end{enumerate}
\end{prop}

\begin{proof}
Let $\eta=u|_{\S}$, 
\be
\begin{split}
\p_\nu|\Na u|=&\frac{\Na \Na u(\Na u, \nu)}{|\Na u|}\\
=&\frac{\nu (u)}{|\Na u|}\Na \Na u (\nu, \nu)+\frac{1}{|\Na u|}\Na \Na u (\Na_\S \eta, \nu)\\
\end{split}
\ee

\vh

Using $ \Delta_Mu = -K |\Na u| $, we have
\be
\nabla \Na u ( \nu, \nu) = \Lp_M u- H \nu(u)  - \Delta_{\S} \eta= -K|\Na u|- H \nu(u) - \Delta_{\S} \eta. 
\ee
We also have, 
\be
\begin{split}
\Na \Na u (\Na_\S \eta, \nu)=&(\Na_\S \eta)(\nu(u))-(\Na_{\Na_\S \eta} \nu)(u)\\
%=&(\Na_\S \eta)(\nu(u))-\la \Na_{\Na_\S \eta} \nu, \Na u \ra\\
=&(\Na_\S \eta)(\nu(u))-\la \Na_{\Na_\S \eta} \nu, \Na_\S \eta \ra+\nu(u) \la \Na_{\Na_\S \eta} \nu, \nu \ra\\
=&(\Na_\S \eta)(\nu(u))-\Pi (\Na_\S \eta, \Na_\S \eta),\\
\end{split}
\ee

\vh

Thus, we have, 
\be\label{MM}
\begin{split}
 \p_\nu | \nabla u  |= & \ -\frac{\Pi (\Na_\S \eta, \Na_\S \eta)}{|\Na u|}    +  \frac{ (\Na_\S \eta) (\nu (u))}{|\Na u|} \\
 & -K\nu(u)- H \frac{|\nu (u)|^2}{|\Na u|} - \frac{\nu(u)}{|\Na u|}\Delta_{_\S} \eta . 
\end{split}
\ee
Hence,  if $u=constant$ on $\Sigma$, 
\be\label{MM2}
\begin{split}
 \p_\nu | \nabla u  |= -K\nu(u)- H |\Na u|. 
\end{split}
\ee
If $\p_{\nu}u=0$ on $\S$, 
\be\label{MM3}
\begin{split}
\p_\nu | \nabla u  |= & \ -\frac{\Pi (\Na_\S \eta, \Na_\S \eta)}{|\Na u|}\\
=& -|\Na u| \Pi (\frac{\Na u}{|\Na u|}, \frac{\Na u}{|\Na u|}). 
\end{split}
\ee
\end{proof}

From Proposition \ref{BoundaryTerms} above, we have on $\p M$, 

\be
\begin{split}
\int_{\p M} \p_{\nu} |\Na u| = \int_{F} -|\Na u| \Pi (\frac{\Na u}{|\Na u|}, \frac{\Na u}{|\Na u|})+\int_{T\cup B} -K\nu(u)- H |\Na u|.  
\end{split}
\ee

In particular,  on $\p M$, 
\be
\begin{split}
|\p_{\nu} |\Na u| |\leq C(||g||_{C^1}+||k||_{C^{0}})|\Na u|.  
\end{split}
\ee

Therefore, for the well-definedness of $\int_{\p M} \p_{\nu} |\Na u|$, it suffices to check if $|\Na u|$ is integrable on $\p M$.  Let $p \in \bar E$, w.l.o.g., identified as $0$ in a local coordinate chart.  From the fact that $u\in C^{0,\alpha}(M)$, apply $W^{2,p}$ estimate followed by Sobolev embedding onto $w_r(x)$, where $r>0$ fixed, in a (conic) annulus $A(1)$ around $p$, we have 
\be 
|\Na u|_{C^{0}(A(r))}\leq |\Na \Na u|_{L^p(A(r))}\leq C r^{\alpha-1}. 
\ee 
$|\Na u|$ is therefore integrable on $\p M$ and also on $M$.  

\

{\bf 2.} \underline{$\int_{\p M_r} \p_{\nu_r} |\Na u| \to \int_{\p M} \p_{\nu} |\Na u|.$} First, let's consider the convergence along the horizontal edges by a blow-up argument.  
\begin{prop}\label{EH}
Let $W$ be a compact neighbourhood along the interior of an horizontal edge $E_H$ and $r(p)=dist (p, E_H)$.  We have $r|\Na \Na u|\xrightarrow[]{} 0$ uniformly in $W$ as $r\to 0$.  
\end{prop}
\begin{proof}
Assume on the contrary that $r|\Na \Na u|$ does not converge to 0 uniformly in $W$ as $r\to 0$. Hence, there exists $\eps_0>0$ and a sequence $\{p_i\}$ in $W$ with 
\be\label{contradiction}
L_i |\Na \Na u|(p_i)\geq \eps_0,
\ee
where $L_i=r(p_i)$ and $L_i \to 0$. 

Denote the point on $E_H$ closest to $p_i$ by $q_i$. 
Let $p$ denote a subsequential limit of $p_i$ and hence also the limit of $q_i$. W.l.o.g., we still denote the subsequence by $\{p_i\}$.  

Then, on a ball (intersecting with a wedge) denoted by $B_1$, define a sequence of functions $u_i$ by scaling around each $q_i$, that is,
\be \label{scaling}
u_i(x):= \frac{u(q_i+L_i x)}{L_i}.
\ee 

Check that for each $i$, 
\begin{enumerate}
\item $u_i(0)=0$, 
\item $\p u_i(x)=(\p u) (q_i+L_i x)$, and 
\item$\p\p u_i(x)=L_i \p\p u(q_i+L_i x)$. 
\end{enumerate} Thus,  
\be \label{scaledPDE}
\Lp_i u_i (x)=L_i \Lp_i u(q_i+L_i x)= - L_i K(q_i+L_i x) |\nabla_i u (q_i+L_i x)| 
\ee since $u$ is spacetime harmonic, 
where $\Na_i$ and $\Lp_i$ respectively mean the connection and the Laplace operator with respect to the pull-back metric $g_i$ from $\phi_i: B_1\to (\mathcal{N}_{L_i}(q_i)\subset M, g)$.  Note that, $g_i\to g(p)$.   

\

For regularity of $u$ on $W$, one reflects the domain along the corresponding side face, then by \cite{Trudinger} and \cite{LiWang}, we know that $u$ is uniformly Lipschitz on $W$.  

Then for $u_i$, since $u$ is uniformly Lipschitz, $u_i \to v$ in some $C^{0,\alpha}$ norm, while $v$ itself is still a Lipschitz function. Moreover, R.H.S of \eqref{scaledPDE} $\to 0$ as $L_i \to 0$.  
Furthermore, $u_i \to v$ in $C^{2}_{loc}$ away from the edge since $u$ itself is $C^{2,\alpha}_{loc}$ away from the edge. 

Therefore, we have $v$ satisfying 
$\Lp_{g_{Euc}} v = 0$ and mixed boundary condition on a model wedge with angle $\theta$.  
Furthermore, from \eqref{contradiction} there exists a point $y$ with distance 1 away from $p$ such that, 
\be\label{contradiction2}
|\p \p v(y)| \geq \eps_0. 
\ee

\

There are 2 cases.  First, if $\theta$ is less than $\pi/2$, then in $M$, $p$ lies on a segment where the dihedral angle is less than $\pi/2$.  Hence, there exists an open neighbourhood $U$ of $p$ which 
\begin{enumerate}
\item sits along a segment where the dihedral angle is less than $\pi/2$, 
\item contains a compact set $V$ containing $p_i$ for all large $i$, and hence $p$. 
\end{enumerate}

By \cite{AK80}, we know that $u\in C^{1,\alpha}(V)$, then by Schauder estimate on $w_r(x)$ in an annulus $A(1)$ around $p$, we get $r|\Na \Na u|\leq C(V)r^{1+(1+\alpha)-2}\leq C(V)r^{\alpha}$ 
$\to 0$ as $r\to 0$.  This contradicts \eqref{contradiction}. 

While for the second case,  if $\theta=\pi/2$, then $\p \p v$ should vanish as $v$ should be linear by boundary Harnack inequality.  This contradicts \eqref{contradiction2}. 
\end{proof}

%Therefore, along the horizontal edges, by Proposition \ref{EH} we have $r|\Na \Na u|=o(1)$.  

What follows is based on a remark in the proof of Theorem 1.4 in \cite{Li0}. Meanwhile,  for the vertical edges of $F$ along which $u$ is $C^{1,\alpha}_{loc}$, when Schauder estimate is applied on $w_r(x):=u(rx)-u(0)$ in $A(r) \subset \ti \O$, a compact neghbourhood along the segment,  we have $|\Na \Na u|_{C^{0}(A(r))}\leq  C(\ti \O) r^{\alpha-1}$.  From this, the integrability of Gauss curvature and geodesic curvature for each level set also follows.  With Gauss-Bonnet Theorem, it will be useful to show Lemma \ref{cubic integration formula} which connects energy conditions with dihedral angles. 

\

We then consider for $p$ being a vertex, by Schauder estimate on $A(r)$ around $p$, we have 
\be 
|\Na \Na u|_{C^{0}(A(r))}\leq C r^{\alpha-2}. 
\ee 

Note that as $\p M_r$ approaches the vertices, the difference in the area is of order $r^2$.  To sum up, as $r\to 0$,
\be \label{boundary integral convergence}
\begin{split}
\int_{\p M_r} \p_{\nu_r} |\Na u| \to \int_{\p M} \p_{\nu} |\Na u|. 
\end{split}
\ee

{\bf 3.} \underline{R.H.S. of \eqref{divergence}.}
\be
\begin{split}
&\Lp |\Na u|\\
=&\frac{1}{|\Na u|}({|\Na \Na u|^2}+\la \Na u, \Na \Lp u \ra -|\Na|\Na u||^2+|\Na u|^2Ric(\nl,\nl)). 
\end{split}
\ee
As $\Lp u= -K|\Na u|$, first note that by Lemma 3.1 in \cite{HKK}, we have
\be
\begin{split}
\Lp|\Na u| \geq &-C(||g||_{C^2}+||k||_{C^1})|\Na u|. 
\end{split}
\ee 
\begin{comment}
On the other hand, by Lemma 3.1 in \cite{HKK}, we have
\be
\begin{split}
&\Psi \\
=&\frac{1}{|\Na u|}(-|\Na u|\la \Na u, \Na K \ra - \la K \Na u, \Na |\Na u| \ra -|\Na|\Na u||^2+|\Na u|^2Ric(\nl,\nl))\\
\geq & -C(||g||_{C^2})|\Na u|
\end{split}
\ee

On the other hand, by Lemma 3.1 in \cite{HKK}, we have
\be
\begin{split}
\Lp|\Na u| \geq -C(||g||_{C^2}+||k||_{C^1})|\Na u|. 
\end{split}
\ee 
\end{comment}
In particular,  
\be
\begin{split}
(\Lp|\Na u|)_- \leq C(||g||_{C^2}+||k||_{C^1})|\Na u|, 
\end{split}
\ee 
i.e.  $(\Lp|\Na u|)_-$ is integrable on $M$.  

\

{\bf 4.} \underline{Conclusion.}
By \eqref{exhaustion} and integrability of $(\Lp|\Na u|)_-$,  we get 
\be
\begin{split}
\int_{M_r} (\Lp|\Na u|)_+= \int_{\p M_r} \p_{\nu_r} |\Na u| + \int_{M_r} (\Lp|\Na u|)_-. 
\end{split}
\ee
\begin{comment}
Then, by Lebesgue dominated convergence theorem applied on the right hand side, 
\be
\begin{split}
\int_{\p M_r} \p_{\nu_r} |\Na u|  + \int_{M_r}(\Lp|\Na u|)_-\to \int_{\p M} \p_{\nu} |\Na u|  + \int_{M} (\Lp|\Na u|)_-.  
\end{split}
\ee
\end{comment}
We can thus by \eqref{boundary integral convergence} and monotone convergence theorem conclude that 
\be
\begin{split}
\int_{\p M} \p_{\nu} |\Na u| = \int_{M} \Lp |\Na u|. 
\end{split}
\ee

\end{proof}

Then by Lemma \ref{PDEcube2}, Lemma \ref{IFC} and Proposition \ref{BoundaryTerms}, we can conclude the following which links energy conditions to dihedral angles.  
\begin{lma}\label{cubic integration formula} Let $(M^3,g,k)$ be an initial data set of type $P$ where the dihedral angles are everywhere less than $\pi$. Further assume that the dihedral angles between $T$ and $F$ and those of $B$ and $F$ are everywhere less than or equal to $\pi/2$.  Let $u$ be a spacetime harmonic function in Lemma \ref{PDEcube2},  we have 
\be\label{CUBE} 
\begin{split}
&\ \int_{M} \frac12  \frac{|\o \Na \o \Na u|^2}{| \nabla u |} + \mu|\Na u| + \la J, \Na u\ra  \, d V\\
&+ \ \int_{\p M} H |\Na u|-\pi(\Na u,\nu) \,d \sigma\\
\leq& \int_{0}^{1} \int_{\S_t} \frac{1}{2}R_{\S_t} dA\,  d t+ \ \int_{0}^1 \int_{\p\S_t}\kappa \, d\tau \,dt \\
= & \int_{0}^{1} \left(  2\pi\chi(\S_t) - \sum_{j=1}^q (\pi-\alpha_j) \right)\,d t.\\
\end{split} 
\ee
where $\S_t=\{u^{-1}(t)\}$, $q$ denotes the number of sides of $P_0$, $\alpha_j$ is the dihedral angle between the edges of the level sets and $H$ is the mean curvature computed with respect to $\nu$, the outward unit normal of $M$. 
\end{lma}

\begin{proof}
First, we know that each regular level set of $u$ must reach the side faces by maximum principle. Moreover, since $u\in W^{3,p}_{loc}(\mathring{M})$,  $C^{2,\alpha}_{loc}$ on each face and $C^{1,\alpha}_{loc}$ around the vertical edges, Sard's theorem is applicable (\cite{dP}). Together with the topology of a prism, each component of a level set of regular values is homeomorphic to $P_0$.  Furthermore, by homogeneous Neumann condition, the dihedral angle of the boundary of the level sets is the same as the dihedral angle between corresponding side faces.  

\

We are going to show that the boundary terms of \eqref{base} actually reveal the boundary energy condition. For $T$ and $B$, on which $|\Na u|$ is nowhere vanishing by maximum principle and $u$ is a constant,  by Proposition \ref{BoundaryTerms}, we have 
\be
\begin{split}
&\ \int_{T\cup B} \p_{\nu} |\Na u| +k(\Na u, \nu) \,d \sigma\\
= & \ \int_{T\cup B} -H |\Na u| -Kg(\Na u, \nu)+k(\Na u, \nu) \,d \sigma\\
= &  \ \int_{T\cup B} -H |\Na u| +\pi(\Na u, \nu) \,d \sigma.\\
\end{split} 
\ee
Then,  on $F$, $\p_{\nu}u=0$. From Proposition \ref{BoundaryTerms} and coarea formula, we have
\be
\begin{split}
&\ \int_{F} \p_{\nu} |\Na u| +k(\Na u, \nu) \,d \sigma\\
=&\ \int_{0}^1 \int_{\p{\S_t}} -\Pi\left( \frac{\Na u}{|\Na u|},\frac{\Na u}{|\Na u|} \right) +k\left(\frac{\Na u}{|\Na u|} ,\nu\right) \,d \tau \,dt\\
= & \ \int_{F} -H |\Na u|+k(\Na^F u, \nu) \,d \sigma+ \ \int_{0}^1 \int_{\p\S_t}\kappa \, d\tau \,dt \\
= &  \ \int_{F} -H |\Na u|+\pi(\Na u, \nu) \,d \sigma+ \ \int_{0}^1 \int_{\p\S_t}\kappa \, d\tau \,dt .\\
\end{split} 
\ee
The proof is then concluded by Lemma \ref{IFC} and Gauss-Bonnet theorem.  
\end{proof}

\section{Dihedral Rigidity}\label{Dihedral Rigidity}
In this section, we are going to use Lemma \ref{cubic integration formula} to show the relation between energy conditions and the geometry and dihedral rigidity of a polyhedral initial data set. 

\subsection{General spacetime case}
First, we consider the case of a type $P$ initial data set in general. 
\begin{lma}\label{PDEimplication}
Let $(M^3,g,k)$ be an initial data set of type $P$, where $P_0$ is a convex $q$-gon, which simultaneously satisfies: 
\begin{enumerate}
\item the dominant energy condition, 
\item $H\geq -tr_{T}k$ on $T$, $H\geq tr_{B}k$ on $B$, 
\item $H\geq |\pi^T(\cdot,\nu)|$ on $F$, where the superscript $^{T}$ means the projection onto the tangent bundle of the corresponding domain and $H$ denotes the mean curvature computed with respect to $\nu$, the outward unit normal of $M$,
\end{enumerate}
where $T, B,$ and $F$ denote the top, the bottom and the side faces of $M$ respectively.  

\

Assume that the dihedral angles are everywhere less than $\pi$; moreover,  the dihedral angles between $T$ and $F$ and those between $B$ and $F$ are everywhere less than or equal to $\pi/2$.  
Let $u$ be a spacetime harmonic function in Lemma \ref{PDEcube2} and $\S_t=\{u^{-1}(t)\}$, where $t\in[0,1]$. Then,  
\begin{enumerate}
\item Let $\{E_j\}_{j=1}^q$ denote the vertical edges of $F$ and $\theta_j:=\sup_{E_j} \alpha_j$, where $\alpha_j$ denotes the dihedral angle between $F$ on $E_j$. Then, $$\sum_{j=1}^q \theta_j \geq (q-2)\pi.$$ 
In particular, the dihedral angles of $M$ cannot be everywhere less than those of $P$.  

\item If the dihedral angles of $M$ are further assumed to be everywhere less than or equal to those of $P$, then
\begin{enumerate}
\item $\mu-|J|=0$ on $M$.  
\item The dihedral angles between $T$ and $F$ and those between $B$ and $F$ are everywhere $\pi/2$.  
\item $M$ is smoothly foliated by $\S_t, t\in[0,1]$. On each $\S_t$, the following properties are satisfied.  
\begin{enumerate}
\item $h_{\Sigma_t}+k_{\Sigma_t}=0$, where $h_{\S_t}$ denotes the second fundamental from of $\S_t$ with respect to $\frac{\Na u}{|\Na u|}$. In particular, $\S_t$ is a free boundary (stable) totally spacetime geodesic MOTS.  
\item $\mu+\la J, \frac{\Na u}{|\Na u|} \ra=0$ and hence $\frac{\Na u}{|\Na u|}=-\frac{J}{|J|}$.  
\item  The dihedral angles of $\S_t$ are all equal to those of $P_0$, $R_{\S_t}=0$ and $\kappa_{\p \S_t}=0$. Therefore, each level set is isometric to $P_0$ up to scaling.  
\end{enumerate}
\item On $\p M$, 
\begin{enumerate}
\item $H= -tr_{T}k$ on $T$, $H= tr_{B}k$ on $B$, 
\item $H= |\pi^T(\cdot,\nu)|$ on $F$. 
\end{enumerate}
\end{enumerate} 
\end{enumerate}
\end{lma}
\begin{proof} 
First, by maximum principle,  $\p_{\nu}u< 0$ on $B$ while $\p_{\nu}u> 0$ on $T$. By Lemma \ref{cubic integration formula}, we have
\be\label{comparison}
\begin{split}
%0\leq 
&\ \int_{M} \frac12  \frac{|\o \Na \o \Na u|^2}{| \nabla u |} + (\mu-|J|)|\Na u|  \,dV \\
&+\ \int_{T} (H+tr_Tk) |\Na u| \,d \sigma+\ \int_{B} (H-tr_Bk) |\Na u| \,d \sigma\\
&+\ \int_{F} (H -|\pi^{T}(\cdot,\nu)|) |\Na u|\ \,d \sigma\\
\leq &\ \int_{0}^1 \int_{\S_t} \frac12  \frac{|\o \Na \o \Na u|^2}{| \nabla u |^2} + \mu + \la J, \frac{\Na u}{|\Na u|}\ra  \, d A \, dt\\
&+\ \int_{T} (H+tr_Tk) |\Na u| \,d \sigma+\ \int_{B} (H-tr_Bk) |\Na u| \,d \sigma\\
&+\ \int_{F} (H |\Na u|-\pi(\Na u,\nu)  \,d \sigma\\
\leq& \int_{0}^{1} \int_{\S_t} \frac{1}{2}R_{\S_t} dA\,  d t+ \ \int_{0}^1 \int_{\p\S_t}\kappa \, d\tau \,dt \\
= &\int_{0}^{1} \left(  2\pi\chi(\S_t) - \sum_{j=1}^q (\pi-\alpha_j) \right)\,d t.\\
\end{split} 
\ee 
By the dominant energy condition and the assumptions on $\p M$, we know that 
\be \label{angle}
\begin{split}\sum_{j=1}^q \theta_j \geq (q-2)\pi.
\end{split} 
\ee
Hence, the dihedral angles of $M$ cannot be everywhere less than those of $P$.  

\

If it is further assumed that the dihedral angles of $M$ are everywhere less than or equal to those of $P$, then the dihedral angles of $\S_t$ are all equal to those of $P_0$ by \eqref{angle}.
Then on $\p M$, we have 
\be \label{BDEC1}
\begin{split}
H= - tr_{T}k \text{ on } T, H= tr_{B}k \text{ on } B \text{ and } H= |\pi^T(\cdot,\nu)| \text{ on } F. 
\end{split} 
\ee

Moreover, on $M$, 
\be \label{DEC0}
\mu-|J|=0, 
\ee 
and
\be \label{spacetimeHessian0}
\bar \Na \bar \Na u = \Na \Na u + |\Na u|k =0. 
\ee

Note that $u$ is not a constant function, hence  $\Na u$ is non-vanishing somewhere.  Furthermore, by Kato's inequailty, on $M$,
\be
\begin{split}
|\Na |\Na u||\leq|\Na \Na u|\leq|k||\Na u|.
\end{split} 
\ee
Then by ODE technique, we know that $\Na u$ is nowhere vanishing and hence each $t\in[0,1]$ is a regular value.  By compactness of $M$, there exists $c\in \R$ such that
\be \label{nonvanishing gradeint}
|\Na u|\geq c >0.
\ee 
Hence, $M$ is a smooth foliation of regular level sets.  Moreover, assume on the contrary that there exists $p$ on the horizontal edges where $T$ or $B$ meet $F$ such that the dihedral angle is less than $\pi/2$, then by \cite{AK80} Proposition (Satz) 3.1, $u$ is $C^{1,\alpha}$ around $p$. Since the dihedral angle is less than $\pi/2$, $\Na u(p)=0$. A contradiction arises.  

\

Then we recall that on each level set $\S_t$,  for $X, Y\in T\S_t$, by \eqref{spacetimeHessian0} we have
\be \label{spacetimegeodesic}
h(X,Y)=\frac{\Na \Na u}{|\Na u|}(X,Y)=-k(X,Y), 
\ee 
where $h$ denotes the second fundamental form for $\S_t$ with respect to $\frac{\Na u}{|\Na u|}$.  
From \eqref{spacetimegeodesic}, $M$ is a smooth foliation of totally spacetime geodesic MOTS. 

\subsection{Stability of free boundary MOTS} \label{StableFBMOTS}
To further study the geometry of each level set and $M$, we have to verify and use the fact that $\S_t$ is a stable free boundary MOTS. 
\begin{prop} (\cite{EHLS} Proposition 2)
For a 2-sided MOTS $\Sigma$, let $\phi\in C^{\infty}(\Sigma)$ and $N$ be a continuous unit normal vector field on $\Sigma$, we have
\be
\begin{split}
&\d_{\phi N}(H+tr_{\S}k)\\
=&-\Lp_{\S} \phi + 2\la W_{\Sigma},\Na^{\S} \phi \ra\\
&+(div_{\S} W_{\Sigma}-|W_{\Sigma}|^2+\frac{1}{2}R_{\Sigma}-\mu-J(N)-\frac{1}{2}|k_{\S}+h_{\Sigma}|^2 )\phi, 
\end{split}
\ee
where $h_{\S}$ and $H$ respectively denote the second fundamental form and mean curvature of $\Sigma$ with respect to $N$ and $W_{\S}\in T\S$ is dual to $k(N,\cdot)|_{T\Sigma}$.  
\end{prop}

A definition of stable capillary MOTS is proposed in \cite{ALY2}, and here we state the free boundary case only. 
\begin{defn}(\cite{ALY2} Definition 5.1)\label{Defcap}
A free boundary MOTS $\Sigma\subset M$ is stable with respect to the variation vector field $X=\varphi N,$ where $N$ is a continuous unit normal vector field on $\Sigma$, if and only if there exists a non-negative function $\varphi\in C^{2}(\Sigma)$, $\varphi\not\equiv 0$ satisfying Robin boundary condition $\frac{\partial \varphi}{\partial \nu}=\Pi(N,N)\varphi$ such that $\delta_X (H+tr_{\S}k)\geq 0$. Moreover, it is called \emph{strictly stably outermost} with respect to the direction $X$ if, moreover, $\delta_X(H+tr_{\S}k)\neq 0$ somewhere on $\Sigma$.
\end{defn}

On $\S_t$, let $N:=\frac{\Na u}{|\Na u|}$. By \eqref{nonvanishing gradeint}, $\frac{1}{|\Na u|}$ is well defined on $M$.  Then, we can consider the flow generated by $\frac{\Na u}{|\Na u|^2}$. For $M$ being foliated by level sets of $u$ and each $\S_t$ is a free boundary MOTS, we have 
\be
\begin{split}
&\d_{\frac{1}{|\Na u|} N}(H_{\S_t}+tr_{\S_t}k)=0.
\end{split}
\ee
While on $\p{\S_t}$, 
\be 
\begin{split}
\p_{\nu}( \frac{1}{|\Na u|})%=&-\frac{\Na_{\nu}|\Na u|}{|\Na u|^2}\\
=&-\frac{\Na_{\nu} \la \Na u,  \Na u \ra}{2|\Na u|^3}\\
%=&-\frac{\Na \Na u ( \Na u, \nu )}{|\Na u|^3}\\
=&\frac{-\la \Na_{\Na u} \Na u,  \nu \ra}{|\Na u|^3}\\
=&\frac{\Pi(N,N)}{|\Na u|}.
\end{split} 
\ee 
Therefore, we can conclude that $\S_t$ is a stable free boundary MOTS. 

\

Following \cite{ALY2} Lemma 5.4 (equation (5.16) in \cite{ALY2}, equation (2.9) in \cite{GS}), the stability of $\S_t$ yields to a positive-semidefinite bilinear form $G$ by integrating $\frac{w^2}{|\Na u|}\d_{|\Na u|N}(H+tr_{\S_t}k)$ over $\S_t$, 
\begin{equation}
G(w,w):=\int_{\Sigma_t} \left(|\nabla^{\Sigma_t} w|^2 +Qw^2\right)dA -\int_{\partial \Sigma_t} \left(\Pi(N,N)-\langle W_{\S_t} ,\nu\rangle\right) w^2 d\tau \geq 0
\end{equation} for all $w\in C^{\infty}(\Sigma_t)$, where $Q:=\frac{1}{2}R_{\Sigma}-\mu-J(N)-\frac{1}{2}|k_{\S}+h_{\Sigma}|^2$. Moreover, if $G(1,1)=0$, then 
\be \label{QW0}
Q=0\,\,\,\text{,}\,\,\,W_{\S_t}=\nabla^{\S_t}\log\frac{1}{|\Na u|} \,\,\,\text{and}\,\,\,\Pi(N,N)=\la W_{\S_t},\nu \ra.  
\ee %where $\varphi$ is the principal eigenfunction.  
 
Now,  by \eqref{comparison}, \eqref{spacetimeHessian0} and \eqref{spacetimegeodesic}, we get that
\be
\begin{split}
0=&\ \int_{\S_t} \frac{|\o \Na \o \Na u|^2}{|\Na u|^2}-\ \frac{1}{2}  R_{\S_t}+\mu + \la J, N \ra  \, d A+ \ \int_{\p \S_t} \Pi(N,N) -\pi(\nu, N) \,d \tau\\
=&\ \int_{\S_t}-\ \frac{1}{2}  R_{\S_t}+\mu + \la J, N \ra +\frac{1}{2}|k_{\S_t}+h_{\S_t}|^2 \, d A+ \ \int_{\p \S_t} \Pi(N,N)-\la W_{\S_t},\nu \ra \,d \tau\\
=&-G(1,1).
\end{split} 
\ee

Then, by \eqref{QW0}, we have 
\be\label{scalarcurvature0}
R_{\S_t}=0.  
\ee 

Let $\tau_t'$ denote the unit tangent vector of $\p\S_t$. From \eqref{BDEC1}, we can see that on $F$

\be \label{Fmean}
H=|\pi^T(\nu,\cdot)|=\sqrt{|\pi(\nu,N)|^2+|\pi(\nu,\tau_t')|^2}, 
\ee
hence, 
\be \label{Fmeang0}
H\geq |\pi(\nu,N)|\geq 0. 
\ee Furthermore, by \eqref{spacetimeHessian0}, on $F$, 
\be \label{F2ndff}
\begin{split}
\Pi(N,N)
=&\la \Na_{\frac{\Na u}{|\Na u|}} \nu, \frac{\Na u}{|\Na u|} \ra\\
%=&\frac{-\la \Na_{\Na u} \Na u,  \nu \ra}{|\Na u|^2}\\
=&\frac{-\Na \Na u (\Na u,  \nu)}{|\Na u|^2}\\
=&k(N, \nu)\\
=&\pi(N,\nu).
\end{split}
\ee 
As a result, 
\be \label{geodesiccurvaturenn}
\begin{split}
\kappa_{\p\S_t}=\Pi(\tau_t',\tau_t')=H-\Pi(N,N)\geq 0. 
\end{split}
\ee  

By \eqref{scalarcurvature0} and Gauss-Bonnet Theorem, we can conclude that
\be
\begin{split}
\kappa_{\p\S_t}=0.
\end{split}
\ee 
Therefore,  $\S_t$ is isometric to $P_0$ up to scaling. Then $g$ can be expressed as $\frac{1}{|\Na u|^2}dt^2+f(t)g_{Euc}$ for some function $f$ which depends on $t$ only. This split form will be useful in Corollary \ref{parabolichyperbolic} for showing dihedral rigidity of parabolic prisms. 
\end{proof}

For $P$ being a rectangular prism (cube), we can further deduce its properties from its symmetry. 
\begin{thm}\label{cuberigidity}
Let $(M^3,g,k)$ be a type $P$ initial data set, where $P_0$ is a rectangle, which simultaneously satisfies: 
\begin{enumerate}
\item the dominant energy condition, 
\item the boundary dominant energy condition, 
\item everywhere the dihedral angle between two faces of $M$ is less than or 
equal to $\pi/2$. 
\end{enumerate}
Let  $\S_t=\{u^{-1}(t)\}$, where $t\in[0,1]$ and $u$ is a spacetime harmonic function solving the mixed boundary problem in Lemma \ref{PDEcube2}.  Then, 
\begin{enumerate}
\item
$M$ is smoothly foliated by $\S_t, t\in[0,1]$. On each $\S_t$, the following properties are satisfied.  
\begin{enumerate}
\item $h_{\Sigma_t}+k_{\Sigma_t}=0$, where $h_{\S_t}$ denotes the second fundamental from of $\S_t$ with respect to $\frac{\Na u}{|\Na u|}$. In particular, 
$\S_t$ is a free boundary (stable) totally spacetime geodesic MOTS.  
\item $\mu+\la J, \frac{\Na u}{|\Na u|} \ra=0$ and hence $\frac{\Na u}{|\Na u|}=-\frac{J}{|J|}$.  
\item The 4 dihedral angles of the edges are all equal to $\pi/2$, $\kappa_{\p \S_t}=0$ and $R_{\S_t}=0$. Hence, each level set is isometric to a Euclidean rectangle.  
\end{enumerate}
\item $\mu=|J|=0$ on $M$.  
\item On $\p M$, 
\begin{enumerate}
\item $R_{\p M}=0$, $\Pi=k|_{T(\p M)}=0$. where $\Pi$ is the second fundamental form of $\p M$ with respect to the outward normal $\nu$.  Consequently, $H_{\p M}=tr_{\p M} k = |\pi^T(\nu,\cdot)|=0,$ where the superscript $^{T}$ means the projection onto the tangent bundle of the corresponding domain. In particular, $\p M$ is isometric to the boundary of a Euclidean rectangular prism.  
\item $(\Na u) |_{\p M}$ is a parallel vector field,  i.e.  $\Na_X \Na u\equiv 0$ for $X\in T(\p M)$.
\end{enumerate}
\item $(M,g,k)$ can be isometrically embedded into Minkowski space.  
\end{enumerate} 
\end{thm}
\begin{proof} 
Based on Lemma \ref{PDEimplication}, we are going to study the geometry of $M$ further by its symmetry.  We get that the dihedral angles are everywhere $\pi/2$ as in the proof of Lemma \ref{PDEimplication}. Then, by Lemma \ref{PDEcube}, $u\in C^{2,\alpha}(M)\cap W^{3,p}_{loc}(\mathring M)$.  Furthermore, $u\in C^{3,\alpha}_{loc}(\mathring M)$ by \eqref{nonvanishing gradeint}. 

From \eqref{spacetimegeodesic}, on $T$ and $B$, we respectively have $\Pi=h=-k|_{T(\p M)}$ and $\Pi=-h=k|_{T(\p M)}$.  Therefore, when we reverse the identification of $T$ and $B$ and solve for another spacetime harmonic function,  we get that on both $T$ and $B$, 
\be \label{Flat faces}
\Pi=k|_{T(\p M)}=0. 
\ee

And since we can choose $T$, $B$ and $F$ freely,  we actually get that \eqref{Flat faces} holds on all 6 faces.  In particular, the geodesic curvature of all the edges of $M$ vanishes. Moreover, $R_{\p M}=0$ since each face of $\p M$ is a stable free boundary MOTS. Hence, we can further conclude that $\p M$ is isometric to the boundary of a Euclidean rectangular prism $\ti P=[0,a_1]\times[0,a_2]\times[0,a_3]$ for some $a_1, a_2, a_3>0$.  

\

Moreover, by the boundary dominant energy condition, we can conclude that on $\p M$, 
\be \label{BDEC0}
H_{\p M}=tr_{\p M}k=|\pi^T(\cdot,\nu)|=|k^T(\cdot,\nu)|=0. 
\ee 

Let $X, Y\in T(\p M)$, by \eqref{spacetimeHessian0}, \eqref{Flat faces} and \eqref{BDEC0}, we get that on $\p M$, 
\be
\begin{split}
\Na_X \Na_Y u%=&|\Na u|\frac{\Na \Na u (X,  Y)}{|\Na u|}\\
=&-|\Na u|k(X,Y)\\
=&0,
\end{split} 
\ee

\be
\begin{split}
\Na_X \Na_\nu u%=&|\Na u|\frac{\Na \Na u (X,  \nu)}{|\Na u|}\\
=&-|\Na u|k(X,\nu)\\
=&0. 
\end{split} 
\ee

Hence, we can conclude that
\be\label{constant gradient}
\Na u \text{ is a parallel vector field on $\p M$.} 
\ee 

Moreover, if we solve for $u^i$, $1 \leq i\leq 3$,  with the corresponding choice of $B_i\subset\{x^i=0\}$ and $T_i\subset\{x^i=a_i\}$ with $u^i=a_i$ on $T_i$.  It is straight forward to check that $(u^1, u^2, u^3)$ is a coordinate system on a neighbourhood of $\p M$ with corresponding vector fields $\tilde \p_i:=\frac{\Na u^i}{|\Na u^i|^2}$.  

\

While on each $\S_t$, 
\be\label{DEC Level set 0}
\mu+\la J, \frac{\Na u}{|\Na u|} \ra=0. 
\ee
Since we can choose another orientation of cubes, then we have another spacetime harmonic function $w$ such that on $M$,
\be
\la J, \frac{\Na u}{|\Na u|} \ra=\la J, \frac{\Na w}{|\Na w|} \ra=-|J|,  
\ee
Note that $\frac{\Na u}{|\Na u|}$ and $\frac{\Na w}{|\Na w|}$ must be different somewhere and hence nowhere equal on $M$ by the following lemma and ODE technique.  
\begin{lma}\label{l-independent-1}(\cite{T} Lemma 8.1)
Let $X=\nabla u/|\nabla u|$  and let   $Y=\nabla \wt u/|\nabla \wt u|$ where $u$ and $\wt{u}$ are  spacetime harmonic functions, then
$$
|\nabla (|X-Y|^2)|\le 2|k||X-Y|^2.
$$
\end{lma}
\begin{proof}
\be 
\begin{split}
\Na X=&\Na (\frac{\Na u}{|\Na u|})\\
=&\frac{\Na \Na u}{|\Na u|}-\frac{1}{|\Na u|^2}\frac{\Na \Na u( \Na u, \cdot)}{|\Na u|}\Na u\\
=&-k+k(X,\cdot)X,
\end{split}
\ee
i.e. in local coordinates, $\Na_i X^j=-k_i^j+k_{mi}X^mX^j$.
Similarly, 
\be 
\begin{split}
\Na Y=-k+k(Y,\cdot)Y.
\end{split}
\ee
Hence, 
\be 
\begin{split}
\Na (|X-Y|^2)%=&\Na (|X|^2+|Y|^2-2\la X, Y\ra) \\
=&-2\la \Na X, Y \ra - 2\la X, \Na Y \ra  \\
=&2(k(Y,\cdot)-k(X,\cdot)\la X, Y\ra+k(X,\cdot)-k(Y,\cdot)\la X, Y\ra)\\
=&2(1-\la X, Y \ra)k(X+Y,\cdot)\\
=&|X-Y|^2k(X+Y,\cdot).
\end{split}
\ee
And
\be 
\begin{split}
|\Na (|X-Y|^2)|
\leq &|X-Y|^2 |k|(|X|+|Y|)\\
\leq& 2|X-Y|^2||k|.
\end{split}
\ee
\end{proof}
As a result,  on $M$,
\be
\mu=|J|=0.
\ee

Consider the initial data set $(\R^3\setminus \ti P, g_{Euc}, 0)$, by \eqref{Flat faces}, we can identify $\p M$ and $\p \ti P$ to form an initial data set with corners $\p M$,   
\be
(M_1,g_1,k_1)=(M \cup (\R^3 \setminus \ti P), g \cup g_{Euc}, k\cup 0).  
\ee

Note that $\p M$ is isometric to $\p \ti P$ and the dihedral angle is everywhere $\pi/2$. Then, one can take Fermi coordinates or $\{\ti \p_i\}_{i=1}^3$ as aforementioned on $\p M$ so that under this chart, $g_1$ is Lipschitz and $k_1$ is $L^{\infty}$ on $M_1$ while smooth up to $\p M$ and $\p(\R^3\setminus \ti P)=\p \ti P$ respectively.  We see that $M_1$ is $\R^3$ topologically and satisfies $E=|P|=0$.  By \eqref{DEC0} and \eqref{BDEC0},  we can apply Corollary 1.1 in \cite{T} or Section VI of \cite{Shibuya}. Therefore, $(M_1,g_1,k_1)$, in particular $(M,g,k)$, can be isometrically embedded into Minkowski space (as a graph of a linear combination of spacetime harmonic functions).  
\end{proof}

\subsection{$k=g$ hyperbolic space}
For the special case $k=g$, we can conclude the dihedral rigidity for general prisms.  
\begin{defn}\label{parabolic rectangle}
Let $(\HH^3,g_{\HH})$ be the hyperbolic space with sectional curvature $-1$.  Fix the coordinate system $(x^1, x^2 ,x^3)$ such that $g_{\HH}$ takes the form
\be 
g_{\HH}=(dx^1)^2+e^{2x^1}\big( (dx^2)^2+(dx^3)^2\big).
\ee
\end{defn}
\begin{cor}(cf. \cite{Li2} Theorem 2.4)\label{parabolichyperbolic}
Let $(M^3,g,g)$ be an initial data set of type $P$ which simultaneously satisfies: 
\begin{enumerate}
\item the dominant energy condition, 
\item $H\geq \pi^{\perp}(\nu,\cdot)$ on $T$, 
\item $H\geq -\pi^{\perp}(\nu,\cdot)$ on $B$, 
\item $H\geq|\pi^{T}(\nu,\cdot)|$ on $F$,  
\item everywhere the dihedral angles between two faces of $M$ is less than or equal to those of $P$, 
\end{enumerate}
where $T$ and $B$ are identified with the face lying on $\{x^1=0\}$ and $\{x^1=1\}$ respectively\footnote{Note that our identification of ``top" and ``bottom" faces is the reverse of \cite{Li2}.}, $H$ is computed with respect to $\nu$, the unit outward normal of $M$. Then $(M,g,g)$ is isometric to a parabolic prism in $(\mathbb{H}^3,g_{\mathbb{H}})$.  
\end{cor}

\begin{proof}
Let $u$ be the spacetime harmonic function in Lemma \ref{PDEcube2}.  By \eqref{comparison} and $k=g$, we have $H=|\pi^{T}(\nu,\cdot)|=0$ on $F$. Moreover, by \eqref{spacetimeHessian0},  on $F$,
\be\label{vanishing2ndff}
\begin{split}
\Pi(N,N)
=&\la \Na_{\frac{\Na u}{|\Na u|}} \nu, \frac{\Na u}{|\Na u|} \ra\\
%=&\frac{-\la \Na_{\Na u} \Na u,  \nu \ra}{|\Na u|^2}\\
=&\frac{-\Na \Na u (\Na u,  \nu)}{|\Na u|^2}\\
=&g(N, \nu)\\
=&0,
\end{split}
\ee
where $\Pi$ denotes the second fundamental form of $\p M$ with respect to the outward normal $\nu$. 
On the other hand, on each $\S_t$,  let $X\in T{\S_t}$
\be
\begin{split}
&\Na_{X}|\Na u|\\
=&\frac{\Na \Na u (\Na u,  X)}{|\Na u|^2}\\
=&-g(N,X)\\
=&0.\\
\end{split}
\ee

From the proof of Lemma \ref{PDEimplication}, we can first conclude the following. Let  $\S_t=\{u^{-1}(t)\}$, then, % where $t\in[0,1]$ and $u$ is a spacetime harmonic function solving the corresponding mixed boundary problem.  Then, 
\begin{enumerate}
\item
$M$ is smoothly foliated by $\S_t, t\in[0,1]$. On each $\S_t$, the following properties are satisfied.  
\begin{enumerate}

\item $h_{\Sigma_t}+g_{\Sigma_t}=0$, where $h_{\S_t}$ denotes the second fundamental from of $\S_t$ with respect to $N=\frac{\Na u}{|\Na u|}$. In particular, 
$\S_t$ is a free boundary stable totally spacetime geodesic MOTS (horosphere).    
\item The dihedral angles of the edges are all equal to those of $P_0$, $\kappa_{\p \S_t}=0$ and $R_{\S_t}=0$,  in particular, each level set is isometric to $P_0$ up to scaling. 
\item $|\Na u||_{\S_t}$ is constant.  
\end{enumerate}
\item
$R^M=-6$ on $M$.  
\item The dihedral angles between $T$ and $F$ and those between $B$ and $F$ are everywhere $\pi/2$.  
\item On $\p M$, 
\begin{enumerate}
\item $H=-2$ on $T$ and $H=2$ on $B$.
\item $H= |\pi^T(\nu,\cdot)|=0$ on $F$.
\item $\Pi=0$ on $F$. 
\end{enumerate}
\end{enumerate} 
Similar to the proof of Theorem \ref{cuberigidity}, we are going to consider the flow generated by $\frac{\Na u}{|\Na u|^2}$.  Since $|\Na u|$ is constant on each level set $\S_t$,  we can make a change of coordinate to express $g$ in the following form on $M$,  
\be
g=ds^2+ f(s)\delta_{ij} dx^idx^j,  
\ee
where $f$ is a function depending on $s$ only. Then, we consider on each level set with respect to the $\p_s$ direction, 
\be
\begin{split}
-2=H(s)=\frac{1}{f(s)}{\delta}^{ij}\frac{1}{2}\p_s( f(s)\delta_{ij})=\frac{\p_s f(s)}{f(s)}. 
\end{split}
\ee
We have 
\be
g=ds^2+ e^{-2s+C}\delta_{ij} dx^i dx^j. 
\ee
Note that since $\p_s$ is pointing in the decreasing $x^1$ direction, after a change of direction,  we can see that $(M,g,g)$ is isometric to a parabolic prism in $(\mathbb{H}^3,g_{\mathbb{H}} )$.  
\end{proof}

\subsection{$k\equiv0$ Riemmanian/ time-symmetric case}Similarly, for $k\equiv0$, we can use the level set method to conclude the following, extending the result in \cite{CK}. 
\begin{cor}(cf. \cite{Li1} Theorem 1.6)\label{Euclideanprism}
Let $(M^3,g)$ be a manifold of type $P$ which simultaneously satisfies: 
\begin{enumerate}
\item $R\geq 0$ in $M$, 
\item $H\geq 0$ on $\p M$,   
\item everywhere the dihedral angles between two faces of $M$ are less than or equal to those of $P$, \end{enumerate}
where $H$ is computed with respect to the unit outward normal of $M$. Then $(M,g)$ is isometric to a Euclidean prism.  
\end{cor}

\section{Polyhedra and Spacetime Positive Mass Theorem}\label{localized polyhedra SPMT}
As discussed, the dihedral rigidity conjecture and the spacetime positive mass theorem can respectively be regarded as a characterisation of the energy conditions on a compact polyhedron and a non-compact asymptotically flat initial data set.  It is expected that a bridge linking these 2 concepts should exist.

\subsection{Dihedral angle deficit as a localisation of mass}
Indeed, a connection is established in \cite{M2} and \cite{MP} for the time-symmetric $(k\equiv0)$ case.  From this perspective, the cubic rigidity theorem (Theorem \ref{maindihedral}) can be seen as a localisation of the following proposition, which is a direct consequence of the spacetime positive mass theorem (\cite{EHLS},\cite{W},\cite{PT}) (below, $n$ denotes the dimension in which the spacetime positive mass theorem theorem holds), Theorem 1.1 in \cite{MP} and well-definedness of ADM energy-momentum vector by Proposition 4.1 in \cite{Bartnik}.  
\begin{thm}\label{Bridge}(cf. \cite{MP} Theorem 1.2)
Let $(M^n, g, k)$ be an asymptotically flat initial data set satisfying the dominant energy condition.  Let $\{P_k\}$ denote a sequence of Euclidean polyhedra satisfying
conditions in \cite{MP} Theorem 1.1.  Let $\vec{a}=a^i {\p_i}$, where $\sum_{i=1}^n (a^i)^2=1$, then
\be
\begin{split}
\lim_{k\to\infty}\left(-\ \int_{\mF(\p P_k)} H  \,d \sigma+\ \int_{\mF(\p P_k)}\pi(\vec{a},\nu) \,d \sigma+\ \int_{\mE(\p P_k)} ( \alpha - \bar \alpha) \, d \mu \right) \ge 0,
\end{split} 
\ee 
where $\mF(\p P_k)$ and $\mE(\p P_k)$ denote the faces and edges of $P_k$ respectively. 

\

In particular, for any fixed Euclidean polyhedron $P$,

\be
\begin{split}
\lim_{r\to\infty}\left(-\ \int_{\mF(r)} H  \,d \sigma+\ \int_{\mF(r)} \pi(\vec{a},\nu) \,d \sigma+\ \int_{\mE(r)}  ( \alpha - \bar \alpha) \, d \mu  \right) \ge 0,
\end{split} 
\ee 
where $\mF(r)$ and $\mE(r)$ are the faces and edges of the polyhedron $P(r)$ obtained by scaling
$P$ by a large constant factor $r$, $\nu$ is the outward unit normal w.r.t \,$g$ on corresponding faces and, $\alpha$ and $ \bar \alpha$ respectively denote the dihedral angle with respect to $g$ and $g_{Euc}$. 
\end{thm}
\begin{proof}
By \cite{MP} Theorem 1.1, we know that as $k\to \infty$, 
$$c(n)E=-\ \int_{\mF(\p P_k)} H \,d \sigma+\ \int_{\mE(\p P_k)} ( \alpha - \bar \alpha) \, d \mu +o(1).$$
On the other hand, 
$$-c(n)|P |\leq c(n)\la \vec{a}, P \ra =\ \int_{\mF(\p P_k)}\pi(\vec{a},\nu) \,d \sigma+o(1).$$
Therefore, the result follows from the spacetime positive mass theorem.  
\end{proof}
\begin{rem}
As pointed out by \cite{MP} $Remark$ 1.1, $\{P_k\}$ need not be convex. 
\end{rem}

\subsection{Application to the spacetime positive mass theorem } (cf.\cite{Li1} Section 5)
In this section, we observe that if a general version of Lemma \ref{PDEimplication} holds, then we can prove the spacetime positive mass theorem (Theorem \ref{SPMT}) with Lohkamp's construction of $(\mu-|J|_g)>0$-island (Section 2 in \cite{Lohkamp}).  First, here is proposed a general version of Lemma \ref{PDEimplication}. 
\begin{conj}\label{generalPDEimplication} Let $n\geq 3$, $P^n$ be a Euclidean prism ($P_0\times[0,1]^{n-2}$) and $(\Omega^n,g,k)$ be an initial data set admitting a degree one map onto $P$.  
Further assume that $(\Omega,g,k)$ simultaneously satisfies: 
\begin{enumerate}
\item the dominant energy condition, 
\item the boundary dominant energy condition, 
%\item $H\geq -tr_{T}k$ on $T$, $H\geq tr_{B}k$ on $B$, 
%\item $H\geq |\pi^T(\cdot,\nu)|$ on $F$, where the superscript $^{T}$ means the projection onto the tangent bundle of the corresponding domain and $H$ denotes the mean curvature computed with respect to $\nu$, the outward unit normal of $\Omega$, 
\item the dihedral angles of $\O$ are everywhere less than or equal to those of $P$. 
\end{enumerate} 
Then, on $\O$, 
$$\mu-|J|_g=0.$$
\end{conj}
This statement cannot be shown by spacetime harmonic functions since the level set approach only works in 3 dimensional cases and requires vanishing second homology.  Therefore, it would be of interest to seek a (dense) foliation of stable MOTS with boundary by other means under certain assumptions (e.g. \cite{ALY2},\cite{EGM}). 

\begin{defn}(\cite{Lohkamp} Definition 2.8) 
An asymptotically flat initial data set $(M^n, g, k)$ is called a $(\mu-|J|_g)>0$-island if there exists a non-empty open set $U\subset  M$ with compact closure such that 
\begin{enumerate}
\item $\mu-|J|_g>0$ on $U$,  and
\item $(M\setminus U,  g,  k)\equiv (\R^n\setminus B_r(0), g_{Euc},0)$. 
\end{enumerate}
\end{defn}

\begin{prop}
Conjecture \ref{generalPDEimplication} implies the spacetime positive mass theorem. 
\end{prop}
\begin{proof}
Let $(M^n,g,k)$ be an asymptotically flat initial data set satisfying the dominant energy condition.  Assume on the contrary that, w.l.o.g. by \\Christodoulou and O’Murchadha's boost argument (\cite{CO}), $E<0\leq |P|$.  From \cite{Lohkamp} Section 2, one can construct a $(\ti \mu-|\ti J|_{\ti g})>0$-island $(\ti M,\ti g,\ti k)$.  

Then, consider a large scaling of $P^n$ such that $\p P$ can be isometrically embedded into $(\ti M\setminus \ti U,  \ti g,  \ti k)$ and encloses $\ti U$.  Let $\ti \O$ denote the region in $\ti M$ bounded by $\p P$.  $(\ti \O, \ti g, \ti k)$ clearly satisfies the assumptions of Conjecture \ref{generalPDEimplication}, where a degree one map is taking $\ti \O \setminus \ti U$ to $P\setminus \{0 \}$ and $\ti U$ to $\{ 0\}$.  Therefore, particularly, $\ti\mu -|\ti J|_{\ti g}=0$ on $\ti U$. A contradiction arises.  
\end{proof}

\

In particular, for the 3 dimensional case, Theorem \ref{SPMT} can be proved by considering a large cube as follows.  Assume on the contrary that the spacetime positive mass theorem does not hold, as aforementioned, then there exists a $(\ti \mu-|\ti J|_{\ti g})>0$-island $(\ti M,\ti g, \ti k)$. Now, the construction of the generalised exterior region (\cite{HKK} Section 2) can be carried out on $(\ti M,\ti g, \ti k)$. In particular,  there exists $(\hat M, \hat g, \hat k)$ with boundary $\p \hat M$ composed of MOTS and MITS such that $H_2(\hat M, \p \hat M, \mathbb{Z})=0$. Moreover, there exists a non-empty open set $\hat U\subset \hat M$ with compact closure such that 
\begin{enumerate}
\item $\hat \mu-|\hat J|_{\hat g}\geq 0$ on $\hat U$, 
\item $(\hat M\setminus \hat U, \hat g, \hat k)=(\ti M\setminus \ti U,  \ti g,  \ti k) \equiv (\R^n\setminus B_r(0), g_{Euc},0)$.  
\end{enumerate} 
Then, let $P^3$ be a cube. As aforementioned, consider a large scaling of $P$ such that $\p P$ can be isometrically embedded into $(\hat M\setminus \hat U,  \hat g,  \hat k)$ and encloses $\hat U$.  Let $\hat \O$ denote the region in $\hat M$ bounded by $\p P$.  Hence, we have $H_2(\hat \Omega, \p \hat M, \mathbb{Z})=0$. 
The arguments in Appendix \ref{Existence of spacetime harmonic functions} can be modified correspondingly to show the following lemma.  
\begin{lma}\label{PDEcube3}
Let $\vec c \in \R^m$, where $m$ denotes the number of components of $\p \hat M$.  Then there exists a non-negative spacetime harmonic function $u_{\vec{c}} \in C^{2,\alpha}(\hat \O)\cap W^{3,p}_{loc}(int \, \hat \O)$ such that 
\begin{enumerate}
\item $\hat \Lp u_{\vec{c}} + \hat K|\hat \Na u_{\vec{c}}|=0$ in $int\,  \hat \O$,
\item $u_{\vec{c}}=c_i$ on $\p_i \hat M$ for $i=1,... m$, where $\p_i \hat M$ denotes the $i$-th component of $\p \hat M$, 
\item $u_{\vec{c}}=0$ on $B$ and $u_{\vec{c}}=1$ on $T$,
\item $\p_{\nu} u_{\vec{c}} =0$ on $F$, 
\end{enumerate}
where $\hat K:=tr_{\hat g}\hat k$, $T, B, F$ and $\mE$ denote the top, the bottom, the side faces and the edges of $\p \hat \O \cap (\hat M\setminus \hat U)$ respectively. 
\end{lma} 
Furthermore, note that the arguments from \cite{HKK} Section 5 can be applied to $ \hat \O$ as we reduce the mixed boundary value problem into a Dirichlet problem in Appendix \ref{Existence of spacetime harmonic functions}. The following lemma thus  holds. 

\begin{lma}\label{ND}(\cite{HKK} Lemma 5.1)
Let $a_i \in\{-1,1\}$ for $i=1,2,...,m$. There exists a constant $\vec{c}$ such that for each $i$, there exists $y_i \in \p_i \hat M$ with $|\hat \Na u_{\vec{c}}(y_i)|=0$, and $(-1)^{a_i}(\p_\nu u_{\vec{c}})\geq 0$ on $\p_i \hat M$, where $\nu$ is the unit normal pointing out of $\hat M$.
\end{lma}

Then, as in the proof of Theorem \ref{cuberigidity}, we get, 
\be\label{comparisonwithboundary}
\begin{split}
0\leq 
&\ \int_{\hat \O} \frac12  \frac{|\hat {\o \Na} \hat{ \o \Na} u_{\vec{c}}|^2}{| \hat \nabla u_{\vec{c}} |} + (\hat \mu-|\hat J|_{\hat g})|\hat \Na u_{\vec{c}}|  \,dV +\ \int_{T\cup B\cup F}( \hat H -|\hat \pi(\cdot,\nu)|) |\hat \Na u_{\vec{c}}|\,d \sigma\\
\leq& \int_{0}^{1} \int_{\S_t} \frac{1}{2}R_{\S_t} dA\,  d t+ \ \int_{0}^1 \int_{\p\S_t}\kappa \, d\tau \,dt + \sum_{i=1}^m \int_{\p_i \hat M}  \hat H |\p_{\nu} u_{\vec{c}}|-tr_{\p_i \hat M}\hat k(\p_{\nu} u_{\vec{c}}) \, d\sigma\\
%= &\int_{0}^{1} \left(  2\pi\chi(\S_t) - \sum_{j=1}^q (\pi-\alpha_j) \right)\,d t\\
%+ \sum_{i=1}^m \int_{\p_i \hat M}  \hat H |\p_{\nu} u_{\vec{c}}|-tr_{\p_i \hat M}\hat k(\p_{\nu} u_{\vec{c}}) \, d\sigma \\
\leq & \sum_{i=1}^m \int_{\p_i \hat M}  \hat H |\p_{\nu} u_{\vec{c}}|-tr_{\p_i \hat M}\hat k(\p_{\nu} u_{\vec{c}}) \, d\sigma\\
= & 0. 
\end{split} 
\ee 
The last equality follows from Lemma \ref{ND} that we can choose $\vec{c}$ such that $u_{\vec{c}}$ on  
MOTS and MITS components of $\p \hat M$, we would have $\p_{\nu} u_{\vec{c}} \leq 0$ and $\p_{\nu} u_{\vec{c}} \geq 0$ respectively. Therefore, $\hat \mu-|\hat J|_{\hat g}=0$ in $\hat \O$, a contradiction arises.

\section{Charged Riemannian cubes}\label{Charged Riemannian cubes}
In this section, we study dihedral rigidity for charged Riemannian cubes which does not rely on its foliation by level sets but the properties of their augmented electric fields.

\begin{defn}
A Riemannian manifold $(M^3,g)$ augmented by a smooth divergence free electric field $\mE$ is called a charged initial data set. 
\end{defn}

\begin{defn}\label{chargedDEC}
A charged initial data set $(M,g,\mE)$ is said to satisfy the charged dominant energy condition if 
$$R\geq 2|\mE|^2.$$ 
\end{defn}

\begin{defn}\label{chargedBDEC}
A charged initial data set $(M,g,\mE)$ is said to satisfy the charged boundary dominant energy condition if on $\p M$, 
$$H\geq 2\la \mE,\nu \ra,$$ 
where $\nu$ is the unit outward normal. 
\end{defn} 

\begin{defn}(\cite{BHKKZ} Section 3)
The charged Hessian tensor is given by
\be
\hat{\nabla}_{i}\hat \Na_{j} u=\nabla_i\Na_{j}u+\mathcal{E}_i u_j +\mathcal{E}_j u_i -\langle\mathcal{E},\nabla u\rangle g_{ij}. 
\ee
A function $u$ on $M$ is called charged harmonic if
\be\label{fjguru}
\Delta u-\langle\mathcal{E},\nabla u\rangle=0.
\ee
\end{defn}

Then, we can conclude the following by considering a charged harmonic function with suitable boundary conditions prescribed. 
\begin{thm}\label{chargedEuclideanprism2} (Theorem \ref{chargedEuclideanprism})
Let $(M^3,g,\mE)$ be a charged initial data set of type $P$, where $P_0$ is a rectangle, which simultaneously satisfies: 
\begin{enumerate}
\item the charged dominant energy condition, 
\item the charged boundary dominant energy condition,   
\item everywhere the dihedral angles between two faces of $M$ are less than or equal to $\pi/2$, 
\end{enumerate}
where $H$ is computed with respect to $\nu$, the unit outward normal of $M$.  Then,  $(M,g)$ is conformally equivalent to a Euclidean rectangular prism. Furthermore, $(M,g,\mE)$ can be isometrically embedded into the time slice of a Majumdar-Papapetrou spacetime. 
\end{thm}
 
\begin{proof}
First, by Appendix \ref{Existence of spacetime harmonic functions 2}, we can conclude the following lemma. 
\begin{lma}\label{PDEcube3charged}
Given a charged initial data set $(M^3,g,\mE)$ of type $P$, where all dihedral angles are everywhere smaller than $\pi$,  then there exists a non-negative charged harmonic function $u\in C^{0,\alpha}(M) \cap C^{1,\alpha}_{loc}(M\setminus (\bar T \cup \bar B)) \cap C^{2,\alpha}_{loc}(M\setminus \bar E)\cap C^{3,\alpha}_{loc}(\mathring M)$ such that 
\begin{enumerate}
\item $\Delta u-\langle\mathcal{E},\nabla u\rangle=0$ in $\mathring M$,
\item $u=0$ on $B$ and $u=1$ on $T$,
\item $\p_{\nu} u =0$ on $F$,
\end{enumerate}
where $\nu$ denotes the outward unit normal of $\p M$; $T, B, F$ and $E$ denote the top, the bottom, the side faces and the edges of $M$ respectively. 
\end{lma} 

As in \eqref{MM}, using the fact that $u$ is charged harmonic, we get that on $\S:=\p M$, 
\be\label{charge}
\begin{split}
 \p_\nu | \nabla u  |= & \ -\frac{\Pi (\Na_\S \eta, \Na_\S \eta)}{|\Na u|}    +  \frac{ (\Na_\S \eta) (\nu (u))}{|\Na u|} \\
 & +\frac{\nu(u)}{|\Na u|}\la \mE,\Na u \ra- H \frac{|\nu (u)|^2}{|\Na u|} - \frac{\nu(u)}{|\Na u|}\Delta_{_\S} \eta ,
\end{split}
\ee
where $\eta=u|_{\S}$. 
Hence,  if $u=constant$ on $\Sigma$, 
\be\label{charge2}
\begin{split}
 \p_\nu | \nabla u  |=\frac{\nu(u)}{|\Na u|} \la \mE,\Na u \ra- H |\Na u|. 
\end{split}
\ee
And if $\p_{\nu}u=0$ on $\S$, 
\be\label{charge3}
\begin{split}
 \p_\nu | \nabla u  |=& -|\Na u| \Pi (\frac{\Na u}{|\Na u|}, \frac{\Na u}{|\Na u|}).\\
 \end{split}
\ee

From the proof of Lemma \ref{cubic integration formula}, by the following inequality (\cite{BHKKZ} equation (8.7)),   
\be \label{======}
\begin{split}
&\frac{1}{2}\int_{M}\left(\frac{|\hat\nabla \hat \Na u|^2}{|\nabla u|}+(R-2|\mathcal{E}|^2-R_{\S_t})|\nabla u|\right)dV\\
\leq &\int_{\p M}\left(\partial_\nu |\nabla u|-\Delta u\frac{\nu (u)}{|\nabla u|}+2|\nabla u|\langle \mathcal{E},\nu\rangle\right)dA, 
\end{split}
\ee
where $\S_t=\{u^{-1}(t)\}$, we obtain 
\be\label{CUBEcharge2} 
\begin{split}
&\ \int_{M} \frac12  \frac{|\hat \Na \hat \Na u|^2}{| \nabla u |} + (R-2|\mE|^2)|\Na u|  \, d V\\
&+ \ \int_{T\cup B} (H -2\la \mE,\nu \ra) |\Na u| +(\Lp u -\la \mE,\Na u \ra) \frac{\nu(u)}{|\Na u|} \,d \sigma + \ \int_{F} (H -2\la \mE,\nu \ra) |\Na u| \,d \sigma\\
\leq& \int_{0}^{1} \int_{\S_t} \frac{1}{2}R_{\S_t} dA\,  d t+ \ \int_{0}^1 \int_{\p\S_t}\kappa \, d\tau \,dt. 
\end{split} 
\ee

Let $\{E_j\}_{j=1}^4$ denote the vertical edges of $F$ and $\alpha_j$ denotes the dihedral angle between $F$ on $E_j$. By Gauss-Bonnet Theorem, 
\be
\begin{split}
&\ \int_{M} \frac12  \frac{|\hat \Na \hat \Na u|^2}{| \nabla u |} + (R-2|\mE|^2)|\Na u|  \, d V+ \ \int_{\p M} (H -2\la \mE,\nu \ra) |\Na u| \,d \sigma\\
\leq & \int_{0}^{1} \left(  2\pi\chi(\S_t) - \sum_{j=1}^4 (\pi-\alpha_j) \right)\,d t.
\end{split} 
\ee
Then, since $R\geq2|\mE|^2$ and $H\geq 2\la \mE,\nu \ra$, we have 
\be
\begin{split}
\sum_{j=1}^4 \theta_j \geq 2\pi,
\end{split}
\ee where $\theta_j:=\sup_{E_j} \alpha_j$. In particular, the dihedral angles of $M$ cannot be everywhere less than $\pi/2$.  As in the proof of Theorem \ref{cuberigidity}, we can by symmetry choose other orientations to solve for other charged harmonic functions.  From the proof of Lemma \ref{PDEimplication}, the dihedral angles of $P$ are indeed everywhere $\pi/2$.  Meanwhile, Lemma \ref{solvabliity} tells us that $u\in C^{2,\alpha}(M)$. Furthermore, $H=2\la \mE,\nu \ra$ on $\p M$ while $R=2|\mE|^2$ and $\hat \Na \hat \Na u=0$ in $M$. Then, we can tell that the following (\cite{BHKKZ} equation (8.19)) holds by (8.20)-(8.22) in \cite{BHKKZ}, 
\begin{equation}\label{riccie}
R_{ij}-|\mathcal{E}|^2 g_{ij}+\mathcal{E}_i \mathcal{E}_j
+\nabla_i \mathcal{E}_j=0.
\end{equation}

Follow further the arguments in \cite{BHKKZ} Section 8.3, since $\Na \mE$ is symmetric and $H^1(M)=0$, there exists a harmonic function $h$ such that $\mE=\Na h$. Let $\ti g:=f^{4}g$, where $f=e^{\frac{-h}{2}}$.  By considering the curvatures under conformal changes, $u$ being charged harmonic and \eqref{riccie},  we have
\be
\begin{split}
\ti R_{ij}=0 \text{\,\,\,and\,\,\,}  \ti H=0.  
\end{split}
\ee
Note that, the dihedral angles of $(M,\ti g=f^4 g)$ are still $\pi/2$ everywhere. Therefore, by Corollary \ref{Euclideanprism}, $(M, \ti g)$ is isometric to $(\ti P, g_{Euc})$ for some Euclidean rectangular prism $\ti P$.  In particular, $(M, g=e^{2h}g_{Euc},\mE=\Na h)$ can be isometrically embedded into the time slice of a Majumdar-Papapetrou spacetime.
\end{proof}

Similar to Theorem \ref{Bridge}, we have the following for asymptotically flat charged initial data sets when we consider the positive mass theorem with charge (\cite{GHHP},\cite{BHKKZ} and references therein). 
\begin{thm}
Let $(M^3, g, \mE)$ be an asymptotically flat charged initial data set satisfying the charged dominant energy condition.  Let $\{P_k\}$ denote a sequence of Euclidean polyhedra satisfying
conditions in \cite{MP} Theorem 1.1.  Then, 
\be
\begin{split}
\lim_{k\to\infty}\left( -\ \int_{\mF(\p P_k)} H  \,d \sigma+\ \int_{\mF(\p P_k)}2\la \mE,\nu \ra \,d \sigma+\ \int_{\mE(\p P_k)} ( \alpha - \bar \alpha) \, d \mu \right) \ge 0, 
\end{split} 
\ee 
where $\mF(\p P_k)$ and $\mE(\p P_k)$ denote the faces and edges of $\p P_k$ respectively, $\nu$ is the outward unit normal with respect to $g$ on corresponding faces, $\alpha$ and $ \bar \alpha$ respectively denote the dihedral angles with respect to $g$ and $g_{Euc}$. 
\end{thm} 

\appendix
\section{Existence of spacetime harmonic functions (Cubes) }\label{Existence of spacetime harmonic functions}
In this section,  we discuss the existence of solutions to the PDE in Lemma \ref{PDEcube2} when the dihedral angles are $\pi/2$ everywhere.  This illustrates the ideas of reducing a mixed boundary problem to a Dirichlet problem. 

\begin{prop}\label{PDEcube}
Given $([0,1]^3,g,k)$, where all dihedral angles are $\pi/2$,  there exists a non-negative spacetime harmonic function $u\in C^{2,\alpha}([0,1]^3)\cap W^{3,p}_{loc}((0,1)^3)$ such that 
\begin{enumerate}
\item $G_0(u):=\Lp u + K|\Na u|=0$ in $(0,1)^3$,
\item $u=0$ on $B$ and $u=1$ on $T$,
\item $\p_{\nu} u =0$ on $F$,
\end{enumerate}
where $K=tr_gk$, $T, B$ and $F$ denote the top, the bottom and the side faces of the cube respectively and $\nu$ is the outward unit normal of $\p [0,1]^3$. 
%By Kato's inequality, we can further conclude that $u\in W^{3,p}_{loc}((0,1)^3)$. 
\end{prop} 

\subsection{Invertibility of Linear operator}\label{InvertibilityofLaplaceoperator}
First, we are going to show the solvability of certain linear mixed boundary value problems, which will be used in Section \ref{Regularisedoperator} to prove the proposition above.  

\begin{defn}
Let $\mB:=\{w\in C^{2,\alpha}([0,1]^3) \,|\,   \p_{\nu} w=0,\exists\,  C_1, C_2 \in \R \,\,\text{s.t. }\, w=C_1 \, \text{on}\,\, T \, \text{and} \,\, w=C_2 \,\text{on}\, B\}$, which is a Banach space with $C^{2,\alpha}([0,1]^3)$ norm. 
\end{defn}
\begin{defn}
Let $\mB_0=\{w\in \mB \,|\,w=0 \, \text{on}\, \,T \, \text{and} \, \, B\}$, which is also a Banach space with $C^{2,\alpha}([0,1]^3)$ norm. 
\end{defn}

The following lemma is implied by the proof of Section 3 in \cite{CK}. And here we provide an alternative proof by reflection (cf. \cite{Li1} Appendix B, \cite{Nittka1}) which reduces the mixed boundary problem to a Dirichlet boundary problem.  This approach can further be utilised when we study the mixed boundary problem on general prisms. 
\begin{lma}\label{solvabliity}
Given $([0,1]^3,g)$, where all dihedral angles are $\pi/2$. If $X$ is a vector field of regularity $C^{0,\alpha}([0,1]^3)$, then the operator $L:\mB \to C^{0,\alpha}([0,1]^3)$ defined by 
$$L(u)=\Lp (u) + \la X, \Na u \ra$$ is invertible. 
\end{lma}

\begin{proof}
 Let $\phi \in \mB$, with corresponding Dirichlet data on top and bottom. The question can then be reduced to the invertibility of $\Lp:\mB_0 \to C^{0,\alpha}([0,1]^3)$. 
Consider
\begin{enumerate}
\item $L(u)=\Lp (u) + \la X, \Na u \ra=f$ in $(0,1)^3$, where $f\in C^{0,\alpha}([0,1]^3)$,
\item $u=0$ on $B$ and $u=0$ on $T$,
\item $\p_{\nu} u =0$ on $F$,
\end{enumerate}
First,  say $B$ and $T$ are identified with $\{x^3=0 \,|\,(x^1,x^2)\in [0,1]^2 \}$ and $\{x^3=1 \,|\, (x^1,x^2)\in [0,1]^2 \}$.  We can make an even isometric reflection along the one of the side faces $F$. Then,  make another even reflection along one of the longer faces.  Without loss of generality,  the quadruple cube is obtained by reflecting along $\{x^1=1\}$, then  $\{x^2=1\}$, identified by $[0,2]^2\times[0,1]$ and denoted by $Q$. 

Correspondingly,  under the coordinate charts (general Fermi coordinate) introduced in \cite{Li1} Lemma 2.2, the metric components $g^{ij}$, the Christoffel symbols $\Gamma^k_{ij}$,  $X^i$ and $f$ are evenly reflected twice as in \cite{Li1} Appendix B. They then would be denoted by $\tilde g$,  $\tilde \Gamma$, $\tilde X$ and $\tilde f$ respectively.  Note that, since the dihedral angle is everywhere $\pi/2$, on the edges and vertices where doubling takes place,  $\tilde g$ is still a well-defined Lipschitz metric on $Q$.  
Identifying and gluing the faces of $Q$ lying on $\{x^1=0\}$ and $\{x^1=2\}$, then the faces lying on $\{x^2=0\}$ and $\{x^2=2\}$, we have obtained $T^2\times[0,1]=S^1\times S^1\times[0,1]$.  We can see that the component functions of $\tilde g \in C^{0,1}(T^2\times[0,1])$ while $\tilde f \in C^{0,\alpha}(T^2\times[0,1])$ and $\tilde\Gamma, \tilde X \in L^\infty (T^2\times[0,1])$.  

\

Then we consider the following PDE,  
\begin{enumerate}
\item $\tilde \Lp u + \tilde g( \tilde X, \tilde \Na u )=\tilde f$ in $T^2\times(0,1)$, 
\item $u=0$ on $T^2\times\{0\}$ and $u=0$ on $T^2\times\{1\}$. 
\end{enumerate} 

By standard elliptic theory (\cite{GT} Theorem 9.15, Theorem 9.13 and Lemma 9.16), there exists a unique strong solution $v\in W^{2,p}(T^2\times[0,1])$, $p>3$, hence $C^{1,\alpha}(T^2\times[0,1])$. In order to show that this $v$ when restricted to one of the cubes solves the mixed boundary problem.  It suffices to show that $v$ is actually periodically evenly reflected.  Back to $Q$, define a new functions $\hat{v}$ and $\hat{f}$ by reflection as follows
\be
\begin{split}
\hat v(x^1,x^2,x^3)&=v(2-x^1, x^2, x^3), \\
\hat f(x^1,x^2,x^3)&=\tilde f(2-x^1, x^2, x^3),
\end{split}
\ee
note that $\hat f =\tilde f$. 

Now,  we consider the following PDE,  
\begin{enumerate}
\item$ \tilde \Lp u + \tilde g( \tilde X, \tilde \Na u) =\hat f$ in $T^2\times(0,1)$,  
\item $u=0$ on $T^2\times\{0\}$ and $u=0$ on $T^2\times\{1\}$.
\end{enumerate} 
By the symmetry of $\tilde f$, $v\in W^{2,p}(T^2\times[0,1])$ is the unique solution.  On the other hand, by the symmetry of coefficients, obviously $\hat v$ is also a solution.  Hence,  $v=\hat v$.  Hence $v$ is even along the plane $\{x^1=1\}$.  Similarly, we can conclude that $v$ is symmetric along the other side faces. Therefore, we can conclude that $v|_{[0,1]^3}\in W^{2,p}([0,1]^3)$ is a strong solution to the mixed boundary problem.  Moreover, since $\p_{\nu} v=0$ on $F$, $v$ is actually $C^2$ in $T^2\times(0,1)$.  For its regularity up to the boundary, in particular across the closed side faces and edges,  let $\Omega$ be a neighbourhood in $T^2\times[0,1]$ near the boundary.  For $\nu$ on $F$ can be extended as a Fermi coordinate system, consider
\begin{enumerate}
\item$ \tilde g^{ij} v_{ij} + \tilde \Gamma^{a}\p_{a}v + \tilde X^{a}\p_{a} v = \tilde f - \tilde \Gamma^{\nu}\p_{\nu} v - \tilde X^{\nu} \p_{\nu} v $ in $\Omega$,  
\item $v=0$ on $T^2\times\{0\}$, 
\end{enumerate}
where $i,j\in \{1,2,3\}$ while $a,b \in\{1,2\}$ stands for the component perpendicular to $\nu$.  Note that $\p_{\nu} v \in C^{1}(T^2\times(0,1))\cap C^{0,\alpha}(T^2\times[0,1])$ and vanishes along the side faces, hence the right hand side is $C^{0,\alpha}(\bar \Omega)$.  Similarly, we can deal with the edges. Therefore, by \cite{GT} Lemma 6.18, we know that $v\in  C^{2,\alpha}(T^2\times[0,1])$. As a result, $L:\mB \to C^{0,\alpha}([0,1]^3)$ is invertible. 
\end{proof}

\subsection{Regularised operator}\label{Regularisedoperator}
After showing invertibility of linear operators in Lemma \ref{solvabliity}, we are going to show that the mixed boundary problem in Proposition \ref{PDEcube} is solvable by implicit function theorem. Since the operator there is not linearisable, we first have to consider the following regularised operator.

Let $\d\in(0,1)$,  let $G_{\d}$ be a regularised operator defined by $G_{\d}(u):=\Lp u + K\sqrt{ \delta^2+|\Na u|^2}-\delta K$.  We are going to consider the following regularised PDE.  

\begin{enumerate}
\item $G_{\d}(u)=0$ in $(0,1)^3$,
\item $u=0$ on $B$ and $u=1$ on $T$,
\item $\p_{\nu} u =0$ on $F$,
\end{enumerate}
where $K=tr_gk$ and $T, B, F$ denotes the top, the bottom and the side faces of a cube respectively. 

\subsection{Aprori estimates}
Let $u\in \mB$ be a solution to the PDE above.  As in Section \ref{InvertibilityofLaplaceoperator},  by reflection and the fact that $$\Lp u=- K\sqrt{ \delta^2+|\Na u|^2}+\delta K,$$ together with  interpolation inequality (\cite{GT} Lemma 6.35), we get the following estimate 
\be
\begin{split}
||u||_{C^{2,\alpha}([0,1]^3)}\leq C(||u||_{C^0([0,1]^3)}+\delta||K||_{C^{0,\alpha}([0,1]^3)}).%\leq C(||u||_{C^0}+\delta),   
\end{split} 
\ee
where $C$ depends on metric $g$. %and $||K||_{C^{0,\alpha}}$.  
Moreover,  by \cite{GT} Theorem 9.1, we have $||u||_{C^0([0,1]^3)} \leq C(1+||K||_{C^{0}([0,1]^3)})$.  Altogether, we have
\be\label{regularisedschauderestimate}
\begin{split}
||u||_{C^{2,\alpha}([0,1]^3)}\leq C(1+||K||_{C^{0,\alpha}([0,1]^3)}), %\leq C.   
\end{split} 
\ee
which is independent of  $\delta$.  

\subsection{Uniqueness of solution}
\begin{prop}
The solution to the regularised PDE is unique. 
\end{prop}
\begin{proof}
If $u$ and $v$ are solutions to the regularised PDE, then we have 
\begin{enumerate}
\item $ \Lp (u-v) + K\frac{\Na(u+v)}{\sqrt{ \delta^2+|\Na u|^2}+\sqrt{ \delta^2+|\Na v|^2}}\cdot \Na(u-v)=0$ in $(0,1)^3$,
\item $u-v=0$ on $B$ and $u-v=0$ on $T$,
\item $\p_{\nu} (u-v) =0$ on $F$,
\end{enumerate}
Then by maximum principle, $u=v$. 
\end{proof}

\subsection{Linearisation of the regularised operator}\label{Linearisation of the regularised operator}
Let $\phi\in \mB$ with $\phi=0$ on $B$ and $\phi=1$ on $T$. For a fixed $\delta\in(0,1)$, 
we consider a mapping $T_{\delta}:\mB_0 \times [0,1] \to C^{0,\alpha}([0,1]^3)$ defined by 
$$T_{\delta}[u,t]=G_{\delta}(u+t\phi).$$ And for each $t$, let $T^{(1)}_{\delta}|_{[u,t]}:\mB_0 \to C^{0,\alpha}([0,1]^3)$ denote its linearisation in the parameter of $\mB_0$.

Let $A=\{t\in [0,1] \,|\,  \exists\, w\in \mB_0 \, \,\text{such that}\,  T_{\d}[w,t]=0.\}$. Equivalently, if $t\in A$, there exists $u\in\mB$ solving the following mixed boundary problem, 
\begin{enumerate}
\item $G_{\d}(u)=0$ in $(0,1)^3$,
\item $u=0$ on $B$ and $u=t$ on $T$,
\item $\p_{\nu} u =0$ on $F$. 
\end{enumerate}

\

$A$ is non-empty obviously since $T_{\delta}[0,0]=0$.  We first show that $A$ is open as follows.  Let $\bar t \in A$, i.e. there exists $\bar u\in \mB_0$ such that 
$$T_{\delta}[\bar u,\bar t]=0.$$
Consider the linearisation,  
$$T^{(1)}_{\delta}|_{[\bar u,\bar t]}(v)=\Lp v+K\frac{\Na(\bar u+ \bar t \phi)}{\sqrt{\d^2+ |\Na (\bar u + \bar t \phi)|^2}}\cdot\Na v,$$ which is invertible by Lemma \ref{solvabliity}. By implicit function theorem (\cite{GT} Theorem 17.6),  we know that there exists a neighbourhood $\mathcal{N}$ of $\bar t$ in $[0,1]$ such that for all $t\in \mathcal{N}$,  there exists $u_{\delta, t}\in \mB_0$ such that  $T_{\delta}[u_{\delta,  t},t]=0$.

By the estimate \eqref{regularisedschauderestimate} which is independent of $\delta\in(0,1)$ and $t\in[0,1]$, we have that $A$ is closed and hence $A=[0,1]$.  Therefore, for all $\delta\in(0,1)$, there exists $u_{\d}\in \mB$ solving the PDE,  
\begin{enumerate}
\item $G_{\d}(u_{\delta})=0$,
\item $u_{\d}=0$ on $B$ and $u_{\delta}=1$ on $T$,
\item $\p_{\nu} u_{\d} =0$ on $F$. 
\end{enumerate}

Again, by the uniform estimate \eqref{regularisedschauderestimate}, we have $u_{\d} \xrightarrow[\delta\to 0]{C^{2,\beta}([0,1^3])} u\in \mB$ for all $0<\beta<\alpha$,  satisfying 
\begin{enumerate}
\item $\Lp u + K|\Na u|=0$,
\item $u=0$ on $B$ and $u=1$ on $T$,
\item $\p_{\nu} u =0$ on $F$.
\end{enumerate}
Furthermore, by Kato's inequailty and \cite{GT} Theorem 9.19, we get $u\in W^{3,p}_{loc}((0,1)^3)$ and by maximum principle,  $u\geq 0$. 

\section{Existence of spacetime harmonic functions (Prisms)}\label{Existence of spacetime harmonic functions 2}
To prove Lemma \ref{PDEcube2},  regularisation and implicit function theorem used in Appendix \ref{Existence of spacetime harmonic functions} are still the essential tools.  Generally, for $P_0$ being a $q$-gon,  we can locally around each vertical edge form a reflection twice for regularity estimates. (For example, see \cite{Nittka1} and \cite{Nittka2} which apply a bi-Lipschitz mapping locally onto the boundary where Neumann conditions are imposed to get it straightened followed by a reflection.) Therefore, after identifying the side faces where Neumann conditions are imposed,  we can apply standard results in \cite{GT} for estimates on Dirichlet problems.  

\

However, since the angles are no longer necessarily constantly $\pi/2$ and local bi-Lipschitz mappings are applied through identification when an even reflection is carried out twice as in Appendix \ref{Existence of spacetime harmonic functions}, the coefficients could be discontinuous but still uniformly bounded across the edges and vertices. Correspondingly, we need to apply weak solution theory instead and consider different Banach spaces.  $G_{\d}(u)$ is then expressed as $$ div(\Na u)+K\sqrt{\d^2+|\Na u|^2}-\d K$$ whose structure would be preserved under bi-Lipschitz transformation. 

\

For a type $P$ initial data set $(M,g,k)$, we define the following. 
\begin{defn}
Let $\mathring M:= int\, M$, $\ti \mB:=\{w\in W^{1,2}(\mathring M) \,|\,   \p_{\nu} w=0\, \text{on}\,\, F,\exists\,  C_1, C_2 \in \R \,\,\text{s.t. }\, w=C_1 \, \text{on}\,\, T \, \text{and} \,\, w=C_2 \,\text{on}\, B\}$, which is a Banach space with $W^{1,2}(\mathring M)$ norm. 
\end{defn}

\begin{defn}
Let $\ti \mB_0=\{w\in \ti \mB \,|\,w=0 \, \text{on}\, \,T \, \text{and} \, \, B\}$, which is also a Banach space with $W^{1,2}(\mathring M)$ norm. 
\end{defn}

\begin{defn}
Let $\ti\mH_0=\{w\in W^{1,2}(\mathring M) \,|\,w=0 \, \text{on}\, \,T \, \text{and} \, \, B\}$, which is also a Banach space with $W^{1,2}(\mathring M)$ norm. 
\end{defn}

\begin{defn}
Let $\phi\in \ti \mB$ with $\phi=0$ on $B$ and $\phi=1$ on $T$. For a fixed $\delta\in(0,1)$, 
we consider a mapping $\ti T_{\delta}:\ti \mB_0 \times [0,1] \to L^2 \subset \ti \mH_0^*$ defined by 
$$\ti T_{\delta}[u,t]=G_{\delta}(u+t\phi).$$ And for each $t$, let $\ti T^{(1)}_{\delta}|_{[u,t]}:\ti\mB_0 \to L^2$ denote its linearisation in the parameter of $\ti \mB_0$.  
\end{defn}

\subsection{Invertibility of Linear operators} Prior to implicit function theorem, we have to show certain linear operators are invertible. 
\begin{lma}\label{solvabliitydivform}
Given a type $P$ initial data set $(M,g,k)$. If $X$ is a bounded vector field, then the operator $\ti L:\ti \mB \to L^2 \subset \ti \mH_0^*$ defined by 
$$\ti L(u) (v)=\int_{M} -( div (\Na u) + \la X, \Na u \ra ) (v)$$ for $v\in \ti \mH_0$, is invertible. 
\end{lma}

\begin{proof}
First,  an alternative weak maximum principle where we only need to consider $\sup_{T\cup B} u_+$ or $\inf_{T \cup B} u_-$ follows from the proof of \cite{GT} Theorem 8.1.  

Without loss of generality, it suffices to consider $\ti L_0:=\ti L|_{\ti \mB_0}$.  Define a bilinear functional $\mL:\ti \mH_0 \times \ti \mH_0 \to \R$ by 
$$\mL (u,v)=\int_{M} \la \Na u, \Na v \ra - \la X, \Na u \ra v$$ for $(u,v) \in \ti \mH_0 \times \ti \mH_0$. 
As in the proof of \cite{GT} Theorem 8.3, we can see there exists a sufficiently large $\l>0$ such that the bilinear functional $ \mL_{\l}$ defined by $$\mL_{\l} (u,v)=\int_{M} \la \Na u, \Na v \ra - \la X, \Na u \ra v +\l uv$$ for $(u,v) \in \ti \mH_0 \times \ti \mH_0$ is coercive.  Then follow the argument of \cite{GT} Theorem 8.3, by Lax-Milgram and Fredholm alternative we can conclude that for all $\ti f\in \ti \mH_0^*$, there exists $u\in \ti \mH_0$ such that 
$$\mL(u,v)=\ti f(v)$$ for all $v\in \ti \mH_0$. It remains to verify that $u\in \ti \mB_0$.  

In particular, for all $f\in L^2$,  there exists $u\in \ti \mH_0$ such that for all $v\in \ti \mH_0$ ,
$$\int_M(-\Lp u -\la X, \Na u \ra -f) v +\int_{F} v \p_{\nu} u  = 0.$$         
By testing with arbitrary $v\in C^1_0(M)$, we can see that $-\Lp u -\la X, \Na u \ra =f$ almost everywhere. Hence, by testing with arbitrary $v\in \ti \mH_0$, we see that $\p_{\nu} u=0$ on $F$ and thus the invertibility of $\ti L_0$ and that of $\ti L$ are concluded.  
\end{proof}

\subsection{Apriori estimates}\label{weakaprioriestimate}
This section is to show one can obtain a uniform (independent of $\d$) bound on $C^{0,\alpha}(M)$ norm and $W^{1,2}(\mathring M)$ norm for solutions to the regularised PDE, that is $u$ which satisfies the following. 
\begin{enumerate}
\item $G_{\d}(u)(v):=\int_{M}-(div(\Na u)+K\sqrt{\d^2+|\Na u|^2}-\d K)(v)=0$ for all $v\in \ti \mH_0$,
\item $u=0$ on $B$ and $u=1$ on $T$,
\item $\p_{\nu} u =0$ on $F$,
\end{enumerate}
where $K=tr_gk$ and $T, B, F$ denotes the top, the bottom and the side faces of a cube respectively. 

\begin{prop}
There exists $C>0$ such that for all $\delta \in (0,1)$, if $u_{\d}\in C^{0,\alpha}(M)\cap W^{1,2}(\mathring M)$ is a weak solution to the regularised PDE above, then
\be 
\begin{split}
||u_{\delta}||_{C^{0,\alpha}(M)} \leq C. 
\end{split}
\ee
\end{prop} \label{Holder estimate}
\begin{proof}
For all non-negative $v\in \ti \mH_0$, 
\be
\begin{split}
\int_{M} div(\Na u_{\d})v+K(\sqrt{\d^2+|\Na u_{\d}|^2}-\delta)v=0.
\end{split}
\ee
Thus, 
\be
\begin{split}
\int_{M} \Na u_{\d}\cdot \Na v\leq 2 |K|_{C^0}\int_{M} v|\Na u_{\d}|. 
\end{split}
\ee 

Also consider that 
\be
\begin{split}
\int_{M} div(\Na (-u_{\d}))v-K(\sqrt{\d^2+|\Na (-u_{\d})|^2}-\delta)v=0,
\end{split}
\ee

i.e. $-u_{\d}$ solves PDE of the same structure.  Then, we can follow the proof of Theorem 8.1 in \cite{GT} to conclude that $u_{\d}\geq 0$ and
\be \label{WMP} 
\begin{split}
\sup_{M}|u_{\delta}| \leq \sup_{T \cup B} |u_{\delta}|=1.
\end{split}
\ee 
Note that $|K(\sqrt{\d^2+|\Na u_{\d}|^2}-\delta)|\leq |K||\Na u_{\d}|$,  we can see that $G_{\delta}$ is not linear but still satisfies the structural inequality in \cite{GT} Section 8.5, which is required for the global Hölder estimate.  In particular, \cite{GT} Theorem 8.22 and 8.29 still hold and we can conclude the following.  

\be 
\begin{split}
||u_{\delta}||_{C^{0,\alpha}(M)} \leq C (\sup_{M}|u_{\delta}|)\leq C. 
\end{split}
\ee 
\end{proof}

\begin{prop} \label{W12estimate}
Let $u_{\d}\in C^{0,\alpha}(M)\cap W^{1,2}(\mathring M)$ be a weak solution to the regularised PDE above, then 
\be
\begin{split}
||u_{\d}||^2_{W^{1,2}(\mathring M)}\leq & C \left( (||K||^2_{C^0(M)}+1)(vol(M)+1) \right ).
\end{split}
\ee 
\end{prop} 
\begin{proof}
Fix $0<r<<1$ and $\eta \in C^{\infty}(\bar M)$ such that 
\begin{enumerate}
\item 
$\eta(p)=1$ if $dist(p,T)\geq r$, 
\item
$\eta(p)=0$ if $p\in T$, 
\item
$|\Na \eta|\leq \frac{C}{r}$. 
\end{enumerate}
We have on $M$, 
\be
\begin{split}
\eta^2 u_{\d} \Lp u_{\d}=-\eta^2 u_{\d} (K\sqrt{\d^2+|\Na u_{\d}|^2}-\d K).
\end{split}
\ee
First, for the left hand side, consider 
\be
\begin{split}
\int_M \eta^2 u_{\d} \Lp u_{\d}=& -\int_M \Na (\eta^2 u_{\d}) \cdot \Na u_{\d} +\int_M div(\eta^2 u_{\d} \Na u_{\d}) \\
=& -\int_M (2 \eta u_{\d} \Na \eta +\eta^2 \Na u_{\d})\cdot \Na u_{\d} +\int_{\p M} \eta^2 u_{\d} \Na_{\nu} u_{\d}\\
=& -\int_M 2 \eta u_{\d} \Na \eta \cdot \Na u_{\d} +\eta^2 |\Na u_{\d}|^2.\\
\end{split}
\ee
Then, 
\be
\begin{split}
-\int_M \eta^2 |\Na u_{\d}|^2=& \int_M -\eta^2 u_{\d} K(\frac{\Na u_{\d}}{\sqrt{\d^2+|\Na u_{\d}|^2}+\d}\cdot \Na u ) + 2 \eta u_{\d} \Na \eta \cdot \Na u_{\d}. \\
\end{split}
\ee
Therefore, 
\be
\begin{split}
\int_M \eta^2 |\Na u_{\d}|^2\leq & \int_M \frac{1}{C}\eta^2|\Na u_{\d}|^2 + 2C |K|^2 \eta^2 u_{\d}^2 + \frac{1}{C}\eta^2  |\Na u_{\d}|^2+2Cu_{\d}^2|\Na \eta|^2\\
\end{split}
\ee
We thus have, 
\be
\begin{split}
\int_{M\setminus T_r} |\Na u_{\d}|^2\leq& C \int_M |K|^2 \eta^2 u_{\d}^2+ u_{\d}^2|\Na \eta|^2\\
\leq& C \left( (||K||^2_{C^0(M)}+ \frac{1}{r^2})\int_M u_{\d}^2 \right), 
\end{split}
\ee 
where $T_r:=\{p\in M \, |\, dist(p,T)\leq r\}$. 

Note that $w_{\delta}:=u_{\delta}-1$ is a solution to the regularised PDE with homogeneous Neumann condition in $F$, $w_{\d}=0$ on $T$ and $w_{\d}=-1$ on $B$. Then, one can choose a cut-off function being 0 on $B$ to carry out the former computation. To conclude, we have 
\be 
\begin{split}
\int_M |\Na u_{\d}|^2\leq & C \left( (||K||^2_{C^0(M)}+ \frac{1}{r^2})(\int_M u_{\d}^2 + (\int_M u_{\d}^2)^{\frac{1}{2}}+vol(M)+vol(M)^{\frac{1}{2}}) \right )\\
\leq & C \left( (||K||^2_{C^0(M)}+ \frac{1}{r^2})(\int_M u_{\d}^2 +vol(M)+2 \right ). 
\end{split}
\ee 
Now,  a uniform $W^{1,2}(\mathring M)$ norm on $u_{\d}$ can be obtained by Proposition \ref{Holder estimate}. 
\end{proof}

\subsection{Existence and regularity of solution}
For each $\delta$,  consider $\ti T^{(1)}|_{[0,0]}$ as in Section \ref{Linearisation of the regularised operator}, by Lemma \ref{solvabliitydivform}, \cite{GT} Theorem 17.6 and 8.29, there exists $t_{\d}>0$ such that we have a solution to $u_{\d,t_{\d}}\in C^{0,\alpha}(M) \cap W^{1,2}(\mathring M)$ to the regularised PDE for the boundary data on $T$ being $t_{\d}$.  Then,  either by method of continuity (with the uniform estimates independent of $t_{\d}\in[0,1]$ in Section \ref{weakaprioriestimate}) or scaling, $u_{\d}:=\frac{u_{\d,t_{\d}}}{t_{\d}}$ is a solution to the regularised PDE.  As aforementioned, we can see that $u_{\d}\geq 0$ and $||u_{\d}||_{C^{0,\alpha}(M)}$ is uniformly bounded.  

By elliptic regularity theory,  we get that $u_{\delta}$ is $C^{2,\alpha}$ away from the edges.  And for each compact $\Omega \subset M \setminus \bar \mE$, where $\mE$ denotes the edges, by interior estimate and boundary estimate as in Appendix \ref{Existence of spacetime harmonic functions}, $||u_{\d}||_{C^{2,\alpha}(\Omega)}$ is uniformly bounded since $|u_{\d}|\leq 1$.  Furthermore,  by \cite{GL} Theorem 4.1, since the dihedral angle are assumed to be less than $\pi$ everywhere, one can see that $u_{\d}$ is $C^{1,\alpha}$ up to the vertical edges away from $\bar T$ and $\bar B$.  And for each compact $\ti \Omega \subset M \setminus (\bar T \cup \bar B)$, $||u_\d||_{C^{1,\alpha}(\ti \Omega)}$ is uniformly bounded since $u_{\d}$ is uniformly bounded in $C^0(M)$ and $W^{1,2}(\mathring M)$.  Therefore,  taking a convergent subsequence in $C^{0,\beta}(M)$, where $0<\beta<\alpha$, as $\d\to 0$, one can obtain a classical solution $u\in  C^{0,\alpha}(M)\cap C^{2,\alpha}_{loc}(M\setminus \bar \mE)\cap C^{1,\alpha}_{loc}(M \setminus (\bar T \cup \bar B))$ to the PDE in Lemma \ref{PDEcube2}.  Moreover,  $u\in W^{3,p}_{loc}(\mathring M)$ by Kato's inequality and \cite{GT} Theorem 9.19.

\end{document}